\documentclass[12pt,bezier]{article}
\usepackage{times}
\usepackage{booktabs}
\usepackage{pifont}
\usepackage{floatrow}
\usepackage{caption}
\usepackage{mathrsfs}
\usepackage[fleqn]{amsmath}
\usepackage{amsfonts,amsthm,amssymb,mathrsfs,bbding}
\usepackage{txfonts}
\usepackage{graphics,multicol}
\usepackage{graphicx}
\usepackage{color}
\usepackage{caption}
\usepackage{cite}
\usepackage{latexsym,bm}
\usepackage{indentfirst}
\usepackage{color}
\usepackage[colorlinks=true,anchorcolor=blue,filecolor=blue,linkcolor=blue,urlcolor=blue,citecolor=blue]{hyperref}
\usepackage{extarrows}
\usepackage{cite}
\usepackage{latexsym,bm}
\usepackage{mathtools}
\evensidemargin=\oddsidemargin\topmargin=-1.5cm

\newtheorem{thm}{Theorem}[section]

\newtheorem{cla}{Claim}[section]
\newtheorem{rem}{Remark}[section]

\newtheorem{lem}{Lemma}[section]

\newtheorem{conj}{Conjecture}[section]

\theoremstyle{definition}

\addtocounter{section}{0}

\begin{document}
\title{Spectral extrema of graphs with fixed size: cycles and complete bipartite graphs\footnote{Supported by the National Natural Science Foundation of China (Nos. 11971445, 11771141 and 12011530064).}}
\author{{\bf Mingqing Zhai$^{a}$}
, {\bf Huiqiu Lin$^{b}$}\thanks{Corresponding author. E-mail addresses: mqzhai@chzu.edu.cn
(M. Zhai); huiqiulin@126.com (H. Lin); jlshu@math.ecnu.edu.cn (J. Shu).}
, {\bf Jinlong Shu$^{c}$}
\\
{\footnotesize $^a$ School of Mathematics and Finance, Chuzhou University, Chuzhou, Anhui 239012, China} \\
{\footnotesize $^b$ Department of Mathematics, East China University of Science and Technology, Shanghai 200237, China}\\
{\footnotesize $^c$ Department of Computer Science and Technology, East China Normal University, Shanghai 200237, China}}
\date{}

\maketitle {\flushleft\large\bf Abstract:}
Nikiforov [Some inequalities for the largest eigenvalue of a graph, Combin. Probab. Comput. 179--189] showed that if $G$ is $K_{r+1}$-free then the spectral radius $\rho(G)\leq\sqrt{2m(1-1/r)}$,
which implies that $G$ contains $C_3$ if $\rho(G)>\sqrt{m}$.
In this paper, we follow this direction on determining which subgraphs will be contained in $G$ if $\rho(G)> f(m)$,
where $f(m)\sim\sqrt{m}$ as $m\rightarrow \infty$.
We first show that if $\rho(G)\geq \sqrt{m}$, then $G$ contains $K_{2,r+1}$ unless $G$ is a star;
and $G$ contains either $C_3^+$ or $C_4^+$ unless $G$ is a complete bipartite graph,
where $C_t^+$ denotes the graph obtained from $C_t$ and $C_3$ by identifying an edge.
Secondly, we prove that if $\rho(G)\geq{\frac12+\sqrt{m-\frac34}}$,
then $G$ contains pentagon and hexagon unless $G$ is a book;
and if $\rho(G)>{\frac12(k-\frac12)+\sqrt{m+\frac14(k-\frac12)^2}}$, then $G$ contains $C_t$ for every $t\leq 2k+2$.
In the end, some related conjectures are provided for further research.

\vspace{0.1cm}
\begin{flushleft}
\end{flushleft}

\section{Introduction}\label{se1}

Let $\mathcal{F}$ be a family of graphs.
We say a graph $G$ is \emph{$\mathcal{F}$-free} if it does not contain any $F\in {\cal F}$ as a subgraph.
The \emph{Tur\'{a}n number} $ex(n,{\cal F})$ is the maximum possible number of edges in an ${\cal F}$-free graph with $n$ vertices.
Tur\'an type problems aim to study the Tur\'{a}n number of fixed graphs.
So far, there is a large volume of literature on Tur\'an type problems,
see the survey paper \cite{ZD1}.

For a graph $G$, we use $A(G)$ to denote the adjacency matrix.
The spectral radius $\rho(G)$ is the largest modulus of eigenvalues of $A(G)$.
Nikiforov proposed a spectral Tur\'an problem which asks to determine the
maximum spectral radius of an ${\cal F}$-free graph with $n$ vertices.
This can be viewed as the spectral analogue of Tur\'an type problem.
The spectral Tur\'an problem has received a great deal of attention in the past decades.
The maximum spectral radius of various graphs have been determined for large enough $n$,
for example,
$K_{r}$-free graphs \cite{WILF},
$K_{s,t}$-free graphs \cite{BG},
$C_{2k+1}$-free graphs \cite{V9},
induced $K_{s,t}$-free graphs \cite{NTT},
$\{K_{2,3},K_4\}$-minor free and $\{K_{3,3},K_5\}$-minor free graphs \cite{TT}.
For more results in this direction, readers are referred to a survey by Nikiforov \cite{V8}.

The problem of characterizing graphs of given size with maximal spectral radii was initially posed by Brualdi and Hoffman \cite{BR} as a conjecture,
and solved by Rowlinson \cite{PR}.
Nosal \cite{Nosal} showed that every triangle-free graph $G$ on $m$ edges satisfies that $\rho(G)\leq \sqrt{m}$.
Very recently, Lin, Ning and Wu \cite{LNW} slightly improved the bound to $\rho(G)\leq \sqrt{m-1}$ when $G$ is non-bipartite and triangle-free.

If we vary the spectral Tur\'an problem by fixing the number of edges instead of the number of vertices,
then the problem asks to determine the maximum spectral radius of an ${\cal F}$-free graph with $m$ edges.
Nikiforov \cite{V7,V09,V6} focused his attention on this kind of problem.
For a family of graph $\cal F$ and an integer $m$,
let $\mathbb{G}(m,\mathcal{F})$ denote the family of $\mathcal{F}$-free graphs with $m$ edges and without isolated vertices.
In particular, if $\mathcal{F}=\{F\}$, then we write $\mathbb{G}(m,F)$
for $\mathbb{G}(m,\mathcal{F})$.
Nikiforov solved this problem for $F=K_{r+1}$ and $F=C_4$.

\begin{thm} \label{th1.1} (\cite{V7, V09})
Let $G\in \mathbb{G}(m,K_{r+1})$. Then $\rho(G)\leq \sqrt{2m(1-\frac1r)}$,
with equality if and only if $G$ is a complete bipartite graph for $r=2$,
and $G$ is a complete regular $r$-partite graph for $r\geq3$.
\end{thm}

\begin{thm} \label{th1.2} (\cite {V6})
Let $G\in \mathbb{G}(m,C_4)$, where $m\geq10$. Then $\rho(G)\leq\sqrt m$,
with equality if and only if $G$ is a star.
\end{thm}

Theorems \ref{th1.1} and \ref{th1.2}
assert that $G$ contains $C_3$ and $C_4$ if $\rho(G)>\sqrt{m}$.
Based on this observation, we wish to know which graphs will be a subgraph of $G$ if $\rho(G)>f(m)$,
where $f(m)\sim\sqrt{m}$ as $m\rightarrow \infty$.
We obtain several results in this direction.
Let $S_n^k$ denote the graph obtained from $K_{1,n-1}$ by adding $k$ disjoint edges within its independent set.
Let $C_t^+$ denote the graph obtained from $C_t$ and $C_3$ by identifying an edge.
The first result is about $K_{2,r+1}$-free graphs and $\{C_3^+,C_4^+\}$-free graphs.

\begin{thm} \label{th1.3}
Let $G$ be a graph with size $m$. Then the following two statements hold.
\vspace{4bp}

\noindent (i) If $G\in\mathbb{G}(m,K_{2,r+1})$ with $r\geq2$ and $m\geq 16r^2$, then $\rho(G)\leq\sqrt m$,
and equality holds if and only if $G$ is a star.
\vspace{4bp}

\noindent (ii) If $G\in\mathbb{G}(m,\{C_3^+,C_4^+\})$ with $m\geq9$, then $\rho(G)\leq\sqrt m$,
and equality holds if and only if $G$ is a complete bipartite graph or one of $S_7^3$, $S_8^2$ and $S_9^1$.
\end{thm}

It is clear that $\mathbb{G}(m,{C}_4)\subseteq\mathbb{G}(m,K_{2,r+1})$ for any positive integer $r$,
since $C_4\cong K_{2,2}$.
Moreover, note that each $C_3$-free or $C_4$-free graph is also $\{C_3^+, C_4^+\}$-free.
Thus we have $\mathbb{G}(m, {C}_3)\cup\mathbb{G}(m, {C}_4)
\subseteq\mathbb{G}(m,\{C_3^+,C_4^+\}).$ Therefore,
Theorem \ref{th1.3}
generalizes related results on $C_3$-free graphs \cite{Nosal} and $C_4$-free graphs \cite{V6}.

Let us turn our focus to graphs without given cycles.
To state our results, we need the following definition.
Let $S_{n,k}$ be the \emph{complete split graph}, i.e., the graph obtained by joining each vertex of $K_k$ to
$n-k$ isolated vertices.
Clearly, $S_{n,1}$ is a star with $n-1$ edges and $S_{n,2}$ is a \emph{book} with $n-2$ pages.
We have the following theorem on short cycles.

\begin{thm} \label{th1.4}
If either $G\in \mathbb{G}(m,C_5)$ with $m\geq8$ or $G\in \mathbb{G}(m,C_6)$ with $m\geq22$,
then $\rho(G)\leq\frac{{1+\sqrt{4m-3}}}2$,
and equality holds if and only if $G\cong S_{\frac{m+3}2, \ 2}$.
\end{thm}

The extremal question on cycles with consecutive lengths is also an important topic in extremal graph theory.
From \cite{BO},
we know that if $G$ is a graph of order $n$ with $m(G)>\lfloor\frac{n^2}4\rfloor$, then $G$ contains
a cycle of length $t$ for every  $t\leq \lfloor\frac{n+3}2\rfloor$.
Our last result follows the research of Nikiforov on the spectral conditions for the existence of cycles with consecutive lengths.
Nikiforov \cite{V9} proved that
if $G$ is a graph of order $n$ and $\rho(G)>\sqrt{\lfloor n^2/4\rfloor}$,
then $G$ contains a cycle of length $t$ for every $t\leq \frac{n}{320}$.
Recently, Ning and Peng \cite{NP} improved Nikiforov's result by showing $G$ contains a cycle of length $t$ for every $t\leq \frac{n}{160}$.
We prove an edge analogue result as follows.

\begin{thm}\label{th1.5}
Let $G$ be a graph of size $m$.
If $\rho(G)>\frac{{k-\frac12+\sqrt{4m+(k-\frac12)^2}}}2$,
then $G$ contains a cycle of length $t$ for every $t\leq 2k+2$.
\end{thm}

Interestingly, Nikiforov \cite{Niki1} proved that if $G$ is a graph of order $n$ with $\rho(G)>
\frac{{k+\sqrt{4kn+k^2-4k}}}2,$
then $G$ contains an even cycle $C_{t}$ for every even number $t\leq 2k+2$.

\section{Basic lemmas}\label{se2}

We need introduce some notations. For a graph $G$ and a subset $S\subseteq V(G)$,
let $G[S]$ denote the subgraph of $G$ induced by $S$.
Let $e(G)$ denote the size of $G$.
For two vertex subsets $S$ and $T$ of $G$ (where $S\cap T$ may not be empty),
let $e(S, T)$ denote the number of edges with one endpoint in $S$ and the other in $T$.
$e(S, S)$ is simplified by $e(S)$.
For a vertex $v\in V(G)$, let $N(v)$ be the neighborhood of $v$, $N[v]=N(v)\cup \{v\}$ and
$N^2(v)$ be the set of vertices of distance two to $v$.
In particular, let $N_S(v)=N(v)\cap S$ and $d_S(v)=|N_S(v)|$.

It is known that $A(G)$ is irreducible
nonnegative for a connected graph $G$. From the Perron-Frobenius Theorem, there is a unique positive
unit eigenvector corresponding to $\rho(G)$, which is called the \emph{Perron vector} of $G$.

\begin{lem}{\rm(\cite{V6})} \label{le2.1} 
Let $A$ and $A'$ be the adjacency matrices of two connected graphs $G$ and $G'$ on the same vertex set.
Suppose that $N_G(u)\varsubsetneqq N_{G'}(u)$ for some vertex $u$.
If the Perron vector $X$ of $G$ satisfies $X^TA'X\geq X^TAX$,
then $\rho(G')>\rho(G)$.
\end{lem}

A graph is called \emph{2-connected}, if it is a connected graph without cut vertex.
A \emph{block} is a maximal 2-connected subgraph of a graph $G$.
In particular, an \emph{end-block} is a block containing at most one cut vertex of $G$.
Throughout the paper,
let $G^*$ denote an extremal graph with maximal spectral radius in $\mathbb{G}(m,F)$ for each fixed $F$.
Let $\rho^*=\rho(G^*)$ and
let $X^*$ be the Perron vector of $G^*$ with coordinate $x_{v}$ corresponding to the vertex $v\in V(G^*)$.
A vertex $u$ in $G^*$ is said to be an \emph{extremal vertex} and denoted by $u^*$, if $x_u=\max\{x_v~|~v\in V(G^*)\}$.
Let $N_i(u^*)=\{v~|~v\in N(u^*), d_{N(u^*)}(v)=i\}$.

\begin{lem} \label{le2.2}
If $F$ is a 2-connected graph and $u^*$ is an extremal vertex of $G^*$, then the following claims hold.

\noindent{$(i)$} $G^*$ is connected.

\noindent{$(ii)$} There exists no cut vertex in $V(G^*)\backslash \{u^*\}$,
and hence $d(u)\geq2$ for any $u\in V(G^*)\backslash N[u^*]$.

\noindent{$(iii)$} If $F$ is $C_4$-free, then $N(v_1)=N(v_2)$ for any non-adjacent vertices $v_1, v_2$ of degree two.
\end{lem}

\begin{proof} (i) Otherwise, let $H_1$ be a component of $G^*$ with $\rho(H_1)=\rho^*$ and $H_2$ be another component.
Selecting a vertex $v\in V(H_1)$ and an edge $w_1w_2\in E(H_2)$,
let $G$ be the graph obtained from $G^*$ by adding an edge $vw_1$ and deleting an edge $w_1w_2$ with possible isolated vertex.
Note that $vw_1$ is a cut edge in $G$. Then $G\in \mathbb{G}(m,F)$.
Moreover, $\rho(G)>\rho(H_1)=\rho^*$, a contradiction.

(ii) Suppose to the contrary that there exists a cut vertex in $V(G^*)\setminus\{u^*\}$.
Then $G^*$ has at least two end-blocks.
We may assume that $B$ is an end-block of $G^*$ with $u^*\notin V(B)$ and a cut vertex $v\in V(B)$.
Let $G'=G^*+\{u^*w~|~w\in N(v)\cap V(B)\}-\{vw~|~w\in N(v)\cap V(B)\}$.
Note that $F$ is 2-connected. Then $G'$ is $F$-free, and
$${X^*}^T(A(G')-A(G^*)){X^*}=\sum_{u_iu_j\in E(G')}2x_{u_i}x_{u_j}-
\sum_{u_iu_j\in E(G^*)}2x_{u_i}x_{u_j}=\sum_{w\in N(v)\cap V(B)}2(x_{u^*}-x_v)x_w\geq 0.$$
Note that $N_{G^*}(u^*)\varsubsetneqq N_{G'}(u^*)$.
By Lemma \ref{le2.1}, $\rho(G')>\rho^*$, a contradiction.

(iii) Let $v_1,v_2$ be a pair of non-adjacent vertices of $G^*$ with
$N(v_1)=\{u_1,u_2\}$ and $N(v_2)=\{w_1,w_2\}$, where $|N(v_1)\cap N(v_2)|\leq1$.
We may assume that $x_{u_1}+x_{u_2}\geq x_{w_1}+x_{w_2}$.
Define $G''=G^*+v_2u_1+v_2u_2-v_2w_1-v_2w_2$.
Notice that $F$ is $C_4$-free and $v_1,v_2$ are a pair of vertices of degree two with the same neighbourhood in $G''$.
Thus, $G''$ is still $F$-free.
However, by Lemma \ref{le2.1},
we have $\rho(G'')>\rho^*$, a contradiction.
\end{proof}

We simplify $A(G^*)$ by $A$. Then
\begin{eqnarray}\label{eq2.1}
\rho^*x_{u^*}=(AX^*)_{u^*}=\sum_{u\in N_0(u^*)}x_u+\sum_{u\in N(u^*)\setminus N_0(u^*)}x_u.
\end{eqnarray}
Notice that ${\rho^*}^2$ is the spectral radius of $A^2$. Thus,
\begin{eqnarray}\label{eq2.2}
{\rho^*}^2x_{u^*}=\sum_{u\in V(G^*)}a_{{u^*}u}^{(2)}x_u=d(u^*)x_{u^*}+\sum_{u\in N(u^*)\setminus N_0(u^*)}d_{N({u^*})}(u)x_u+\sum_{w\in N^2(u^*)}d_{N({u^*})}(w)x_w,
\end{eqnarray}
where $a_{{u^*}u}^{(2)}$ is the number of walks of length two from ${u^*}$ to the vertex $u$.
Equalities (\ref{eq2.1})
and (\ref{eq2.2}) will be frequently used in the following proof.

\section{Proof of Theorem \ref{th1.3}}\label{se3}

We first consider $K_{2,r+1}$-free graphs.
Let $G^*$ be an extremal graph for $\mathbb{G}(m,K_{2,r+1})$.
Note that $K_{1,m}\in\mathbb{G}(m,K_{2,r+1})$.
We have $\rho^*\geq \rho(K_{1,m})=\sqrt{m}$.
For convenience,
let $U=N(u^*)$ and $W=V(G^*)\setminus N[{u^*}]$, where $u^*$ is an extremal vertex of $G^*$.
Then (\ref{eq2.2}) becomes
\begin{eqnarray}\label{eq3.1}
{\rho^*}^2x_{u^*}=|U|x_{u^*}+\sum_{u\in U}d_{U}(u)x_u+\sum_{w\in W}d_{U}(w)x_w.
\end{eqnarray}

Notice that $G^*$ is $K_{2,r+1}$-free. Then we have the following claim.

\begin{cla} \label{cl3.1}
$d_U(u)\leq r$ for any vertex $u\in U\cup W.$
\end{cla}

Furthermore, we give the following claims.
\begin{cla} \label{cl3.2}
If $e(U)=0$, then $G^*\cong K_{1,m}$.
\end{cla}

\begin{proof} Note that $\rho^*\geq\sqrt{m}$.
By (\ref{eq3.1}), we have
$$mx_{u^*}\leq{\rho^*}^2x_{u^*}\leq(|U|+2e(U)+e(U,W))x_{u^*}.$$
This implies that $e(W)\leq e(U).$
If $e(U)=0$, then $e(W)=0$. Therefore, $G^*$ is a bipartite graph and thus $G^*$ is triangle-free.
By Theorem \ref{th1.1}, $G^*\cong K_{s,t}$ for some $s$ and $t$ with $s\leq t$.
Note that $m\geq 16r^2$. If $s\geq2$, then $G^*$ contains $K_{2,r+1}$ as a subgraph.
Hence $s=1$, as desired.
\end{proof}

Now let $U_1=\{u\in U~|~\sum_{w\in N_W(u)}d_W(w)>r^2\}$ and $U_2=U\setminus U_1$.

\begin{cla} \label{cl3.3}
If $U_1\neq \emptyset$, then $\frac12\sum_{u\in U_1}d_{U}(u)<e(W)$.
\end{cla}

\begin{proof}
Denote $N_W(U_1)=\cup_{u\in U_1}N_W(u)$. Given $w\in N_W(U_1)$.
By Claim \ref{cl3.1}, $d_{U_1}(w)\leq r$, and hence
$d_W(w)$ is counted at most $r$ times in $\sum_{u\in U_1}\sum_{w\in N_W(u)}d_W(w)$.
Thus we have
\begin{eqnarray}\label{eq3.2}
\sum_{w\in N_W(U_1)}d_W(w)\geq \frac 1{r}\sum_{u\in U_1}\sum_{w\in N_W(u)}d_W(w)>\frac 1{r}\sum_{u\in U_1}r^2=r|U_1|.
\end{eqnarray}
Note that $N_W(U_1)\subseteq W$. Thus by (\ref{eq3.2}),
$$e(W)=\frac 12\sum_{w\in W}d_W(w)\geq\frac 12\sum_{w\in N_W(U_1)}d_W(w)>\frac r2|U_1|.$$
It follows from Claim \ref{cl3.1} that
$$\frac12\sum_{u\in U_1}d_{U}(u)\leq \frac r2|U_1|<e(W),$$ as desired.
\end{proof}

Now let $U_2'=\{u\in U_2~|~d_W(u)>\frac {r(\rho^*+2r)}{2(\rho^*-r)}\}$ and $U_2''=U_2\setminus U_2'$.

\begin{cla} \label{cl3.4}
For any $u\in U_2'$, $\frac12d_U(u)x_u+\sum_{w\in N_W(u)}x_w<d_W(u)x_{u^*}.$
\end{cla}

\begin{proof}
Let $u\in U_2'$.
By Claim \ref{cl3.1}, $d_U(u)\leq r$ and $|N_U(w)|\leq r$ for any $w\in N_W(u)$.
By the definition of $U_2$, $\sum_{w\in N_W(u)}|N_W(w)|\leq r^2.$
Thus,
\begin{eqnarray}\label{eq3.3}
\rho^*\sum_{w\in N_W(u)}x_w=\sum_{w\in N_W(u)}\rho^*x_w=\sum_{w\in N_W(u)}\left(\sum_{v\in N_W(w)}x_v+\sum_{v\in N_U(w)}x_v\right)\leq (r^2 +d_W(u)\cdot r)x_{u^*}.
\end{eqnarray}
Note that $d_W(u)>\frac {r(\rho^*+2r)}{2(\rho^*-r)}$.
By (\ref{eq3.3}) we have
$$\frac12d_U(u)x_u+\sum_{w\in N_W(u)}x_w\leq \frac r2 x_{u^*}+\frac r{\rho^*}(r+d_W(u))x_{u^*}<d_W(u)x_{u^*},$$
as claimed.\end{proof}

\begin{cla} \label{cl3.5}
For any $u\in U_2''$, $x_u<\frac12x_{u^*}.$
\end{cla}

\begin{proof}
Let $u\in U_2''$.
By (\ref{eq3.3}), we have
\begin{eqnarray}\label{eq3.4}
\rho^* x_u=x_{u^*}+\sum_{v\in N_U(u)}x_v+\sum_{w\in N_W(u)}x_w\leq \left(1+r+\frac r{\rho^*}(r+d_W(u))\right)x_{u^*}.
\end{eqnarray}
By the definition of $U_2''$, we have $d_W(u)\leq\frac {r(\rho^*+2r)}{2(\rho^*-r)}.$
Thus by (\ref{eq3.4}), we have
\begin{eqnarray}\label{eq3.5}
\rho^* x_u\leq \left(1+r+\frac {3r^2}{2(\rho^*-r)}\right)x_{u^*}.
\end{eqnarray}
In order to show that $x_u<\frac12x_{u^*},$ it suffices to show that $1+r+\frac {3r^2}{2(\rho^*-r)}<\frac{\rho^*}2$,
or equivalently, ${\rho^*}^2-(3r+2)\rho^*-(r^2-2r)>0.$
Since $r\geq2$ and $\rho^*\geq\sqrt{m}\geq4r$, we have
\begin{eqnarray}\label{eq3.6}
{\rho^*}^2-(3r+2)\rho^*-(r^2-2r)\geq 3r(r-2)\geq0.
\end{eqnarray}
And if equality in (\ref{eq3.6}) holds,
then all the above inequalities will be equalities.
In particular, (\ref{eq3.3}) becomes an equality.
This implies that $x_v=x_{u^*}$ for each $w\in N_W(u)$ and $v\in N_U(w)$.
Note that $u\in N_U(w)$ for each $w\in N_W(u)$.
Then $x_u=x_{u^*}$.
However,
(\ref{eq3.5}) and (\ref{eq3.6})
imply that $x_u\leq \frac12x_{u^*}$,
a contradiction.
Hence, (\ref{eq3.6}) is a strict inequality.
The claim follows.
\end{proof}

In the following, we give the proof of Theorem \ref{th1.3}~(i).

\begin{proof}
Note that $\rho^*\geq \sqrt{m}$.
According to (\ref{eq3.1}), we have
\begin{eqnarray}\label{eq3.7}
(m-|U|)x_{u^*}\leq\sum_{u\in U}d_{U}(u)x_u+\sum_{w\in W}d_{U}(w)x_w,
\end{eqnarray}
where $$\sum_{u\in U}d_{U}(u)x_u=\frac12\sum_{u\in U_1\cup U_2'}d_{U}(u)x_u+\sum_{u\in U_2''}d_{U}(u)x_u+\frac12\sum_{u\in U_1\cup U_2'}d_{U}(u)x_u.$$
On one hand, by Claim \ref{cl3.5}, $x_u<\frac12x_{u^*}$ for $u\in U_2''$.
Thus, 
\begin{eqnarray}\label{eq3.8}
\frac12\sum_{u\in U_1\cup U_2'}d_{U}(u)x_u+\sum_{u\in U_2''}d_{U}(u)x_u\leq \frac12\sum_{u\in U}d_{U}(u)x_{u^*}=e(U)x_{u^*},
\end{eqnarray}
where inequality is strict if $U_2''\neq \emptyset$ and $e(U_2'',U)\neq0$.
On the other hand,
by Claim \ref{cl3.3} and Claim \ref{cl3.4},
\begin{eqnarray}\label{eq3.9}
\frac12\sum_{u\in U_1\cup U_2'}d_{U}(u)x_u\leq e(W)x_{u^*}+\sum_{u\in U_2'}\left(d_W(u)x_{u^*}-\sum_{w\in N_W(u)}x_w\right),
\end{eqnarray}
where inequality is strict if $U_1\cup U_2'\neq \emptyset.$
Moreover,
$$\sum_{u\in U_2'}\left(d_W(u)x_{u^*}-\sum_{w\in N_W(u)}x_w\right)\leq\sum_{u\in U}\left(d_W(u)x_{u^*}-\sum_{w\in N_W(u)}x_w\right)
=e(U,W)x_{u^*}-\sum_{w\in W}d_{U}(w)x_w.$$
Combining with (\ref{eq3.9}), we have
\begin{eqnarray}\label{eq3.10}
\frac12\sum_{u\in U_1\cup U_2'}d_{U}(u)x_u\leq (e(W)+e(U,W))x_{u^*}-\sum_{w\in W}d_{U}(w)x_w.
\end{eqnarray}
Combining (\ref{eq3.8}), (\ref{eq3.10}) with (\ref{eq3.7}),
we have $$(m-|U|)x_{u^*}\leq(e(U)+e(W)+e(U,W))x_{u^*}=(m-|U|)x_{u^*}.$$
This implies that, if $m\neq|U|$, then both (\ref{eq3.8}) and (\ref{eq3.9}) are equalities. 
Recall that if (\ref{eq3.9}) is an equality then $U_1\cup U_2'=\emptyset$ (i.e., $U=U_2''$),
and if (\ref{eq3.8}) is an equality then either $U_2''=\emptyset$ or $e(U_2'',U)=0$.
Note that $U\neq\emptyset$. It follows that $U_2''=U\neq \emptyset$ and thus $e(U)=0$.
By Claim \ref{cl3.2}, we have $G^*\cong K_{1,m}$.
If $m=|U|$, we also have $G^*\cong K_{1,m}$.
This completes the proof.
\end{proof}

In what follows, we consider $\{C_3^+,C_4^+\}$-free graphs.
Before the final proof, we shall give some lemmas on characterizing the extremal graph $G^*$.
Recall that $S_n^k$ is
the graph obtained from $K_{1,n-1}$ by adding $k$ disjoint edges within its independent set.

\begin{lem} \label{le4.1}(\cite {V6})
$\rho(S_n^k)$ is the largest root of the polynomial
$$f(x)=x^3-x^2-(n-1)x+n-1-2k.$$
Moreover, $\rho(S_n^1)>\sqrt{n}$ for $4\leq n\leq 8.$
\end{lem}

\begin{rem}\label{re2.1}
Note that $S_m^1$ is a $\{C_3^+, C_4^+\}$-free graph of size $m$.
Lemma \ref{le4.1} implies that the condition $m\geq9$
in Theorem \ref{th1.3}~(ii) is best possible.
\end{rem}

Lemma \ref{le4.1} also implies the following result.

\begin{lem} \label{le4.2}
$\rho(S_n^k)=\sqrt{e(S_n^k)}=3$ for $(n,k)\in\{(9,1),(8,2),(7,3)\}$.
\end{lem}

\begin{proof}
For $(n,k)\in\{(9,1),(8,2),(7,3)\}$, we have $e(S_n^k)=9$ and $k=10-n$.
Let $\rho=\rho(S_n^k)$. By Lemma \ref{le4.1}, we have
$$\rho^3-\rho^2-(n-1)\rho+3(n-7)=(\rho-3)(\rho^2+2\rho-n+7)=0.$$
Note that $n\leq9$. The largest root of above polynomial is $\rho=3$, as desired.
\end{proof}

Recall that $N_i(u^*)=\{u~|~u\in N(u^*), d_{N(u^*)}(u)=i\}$.
Since $G^*$ is $C_3^+$-free, we can see that $G^*[N(u^*)]$ may consist of isolated vertices and isolated edges.
Thus, $N(u^*)=N_0(u^*)\cup N_1(u^*)$.
Let $N_i^2(u^*)=\{w~|~w\in N^2(u^*), d_{N_i(u^*)}(w)\geq1\}$ for $i\in\{0,1\}$.
Note that $G^*$ is $\{C_3^+, C_4^+\}$-free.
Then $N_0^2(u^*)\cap N_1^2(u^*)=\emptyset$ and $d_{N(u^*)}(w)=1$ for any $w\in N_1^2(u^*).$
Since $G^*$ is an extremal graph, $\rho^*\geq\rho(K_{1,m})=\sqrt{m}$.
Combining with (\ref{eq2.2}),
we have
\begin{eqnarray}\label{eq4.1}
(m-d(u^*))x_{u^*}\leq\sum_{u\in N_1(u^*)}x_u+\sum_{w\in N_1^2(u^*)}x_w+\sum_{w\in N_0^2(u^*)}d_{N(u^*)}(w)x_w.
\end{eqnarray}
By (\ref{eq4.1}) and the definition of $u^*$, we have
\begin{eqnarray}\label{eq4.2}
(m-d(u^*))x_{u^*}\leq(2e(N_1({u^*}))+e(N^2({u^*}), N({u^*})))x_{u^*}.
\end{eqnarray}
Let $W=V(G^*)\setminus N[u^*]$. Note that $e(N(u^*))=e(N_1(u^*)).$
(\ref{eq4.1}) and (\ref{eq4.2}) imply that
\begin{eqnarray}\label{eq4.3}
e(W)\leq e(N_1({u^*})),
\end{eqnarray}
and if $e(W)=e(N_1({u^*}))$, then $x_u=x_{u^*}$ for any vertex $u\in N_1({u^*})\cup N^2({u^*})$.

\begin{lem} \label{le4.3}
If $m\geq13$, then $e(W)=0$.
\end{lem}

\begin{proof}
If $N_1(u^*)=\emptyset$, then by (\ref{eq4.3}), $e(W)=0$.
In the following, let $N_1(u^*)\neq\emptyset$ and $E(G^*[N_1(u^*)])=\{u_{2i-1}u_{2i}~|~i=1,2,\ldots,t\}.$
Recall that $G^*[N_1(u^*)]$ consists of isolated edges.
For each vertex $u_i\in N_1(u^*)$,
let $\phi(u_i)$ be the number of paths $u_iww'$ with $w,w'\in W$.
Clearly, $\phi(u_i)=\sum_{w\in N_W(u_i)}d_W(w).$
Furthermore, we can divide $\{1,2,\ldots,t\}$ into two subsets $A_1$ and $A_2$,
where $$A_1=\{i~|~\phi(u_{2i-1})+\phi(u_{2i})\leq1\},~~ A_2=\{i~|~\phi(u_{2i-1})+\phi(u_{2i})\geq2\}.$$
Note that $G^*$ is $C_3^+$-free.
Thus, $N_W(u_{2i-1})\cap N_W(u_{2i})=\emptyset$ for $i\in \{1,2,\ldots,t\}$, and hence
\begin{eqnarray}\label{eq4.4}
e(W)\geq \frac12\sum_{i\in A_2}\left(\phi(u_{2i-1})+\phi(u_{2i})\right)\geq|A_2|.
\end{eqnarray}
Moreover, we have the following claim.

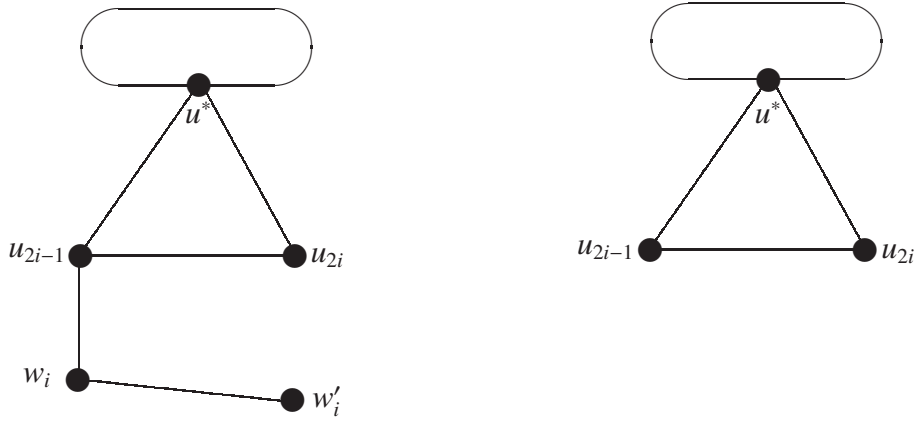
\begin{figure}[t]
\centering \setlength{\unitlength}{2.4pt}
\begin{center}
\unitlength 0.6mm 
\linethickness{0.4pt}
\ifx\plotpoint\undefined\newsavebox{\plotpoint}\fi 
\begin{picture}(186.75,89.3)(0,0)
\multiput(33,71.775)(-.03370098039,-.04840686275){816}{\line(0,-1){.04840686275}}
\put(6.5,33.275){\circle*{5}}
\multiput(33.25,71.525)(.0337078652,-.0613964687){623}{\line(0,-1){.0613964687}}
\put(6.5,33.275){\line(1,0){47.75}}
\put(32.75,71.025){\circle*{5}}
\put(54,33.275){\circle*{5}}
\put(32.25,79.525){\oval(51,17)[]}
\put(6.25,33.525){\line(0,-1){29}}
\multiput(6,5.775)(.320469799,-.033557047){149}{\line(1,0){.320469799}}
\put(6,5.775){\circle*{5}}
\put(53.5,1.275){\circle*{5}}
\put(32.75,65){\makebox(0,0)[cc]{$u^*$}}
\put(-3,34.025){\makebox(0,0)[cc]{$u_{2i-1}$}}
\put(61.25,32.775){\makebox(0,0)[cc]{$u_{2i}$}}
\multiput(159.275,73.05)(-.03370098039,-.04840686275){816}{\line(0,-1){.04840686275}}
\put(132.775,34.55){\circle*{5}}
\multiput(159.525,72.8)(.0337078652,-.0613964687){623}{\line(0,-1){.0613964687}}
\put(132.775,34.55){\line(1,0){47.75}}
\put(159.025,72.3){\circle*{5}}
\put(180.275,34.55){\circle*{5}}
\put(158.525,80.8){\oval(51,17)[]}
\put(159,65){\makebox(0,0)[cc]{$u^*$}}
\put(123,34.525){\makebox(0,0)[cc]{$u_{2i-1}$}}
\put(187.75,33.775){\makebox(0,0)[cc]{$u_{2i}$}}
\put(-3,6.275){\makebox(0,0)[cc]{$w_i$}}
\put(61.25,1.025){\makebox(0,0)[cc]{$w_i'$}}
\end{picture}
\end{center}
\caption{\footnotesize{The two cases when $i\in A_1$.}}\label{fig4.1}
\end{figure}

\begin{cla} \label{cl4.1}
$x_{u_{2i-1}}+x_{u_{2i}}<x_{u^*}$ for each $i\in A_1$.
\end{cla}

\begin{proof}
By Lemma \ref{le2.2},
$d(w)\geq2$ for any $w\in W$.
Recall that $d_{N(u^*)}(w)=1$ for any $w\in N_1^2(u^*)$.
Thus, $d_{W}(w)\geq1$ for any $w\in N_1^2(u^*)$.
Given $i\in A_1$. If $\phi(u_{2i-1})+\phi(u_{2i})=1$, then $|N_W(u_{2i-1})\cup N_W(u_{2i})|=1$.
Moreover, $d_W(w_i)=1$ for the unique vertex $w_i\in N_W(u_{2i-1})\cup N_W(u_{2i})$.
We may assume without loss of generality that $N_W(u_{2i-1})=\{w_i\}$ and $N_W(w_i)=\{w_i'\}$
(see Fig.\ref{fig4.1}).
Then $\rho^* x_{u_{2i}}=x_{u^*}+x_{u_{2i-1}}$ and $\rho^* x_{w_i}=x_{w_i'}+x_{u_{2i-1}}\leq \rho^* x_{u_{2i}}$.
Thus, $${\rho^*}^2 x_{u_{2i-1}}=\rho^*(x_{u^*}+x_{u_{2i}}+x_{w_i})\leq(\rho^*+2)x_{u^*}+2x_{u_{2i-1}}.$$
It follows that
$$x_{u_{2i-1}}\leq\frac{\rho^*+2}{{\rho^*}^2-2}x_{u^*}~~ {and} ~~x_{u_{2i}}=\frac1{\rho^*}(x_{u^*}+x_{u_{2i-1}})\leq \frac{\rho^*+1}{{\rho^*}^2-2}x_{u^*}.$$ Since $\rho^*\geq\sqrt{m}\geq\sqrt{13}$, we have $2\rho^*+3<{\rho^*}^2-2$.
Thus,
$x_{u_{2i-1}}+x_{u_{2i}}<x_{u^*}$, as desired.

If $\phi(u_{2i-1})+\phi(u_{2i})=0$, then $|N_W(u_{2i-1})\cup N_W(u_{2i})|=0$
(see Fig.\ref{fig4.1}).
Now $x_{u_{2i-1}}=x_{u_{2i}}$ and $\rho^* x_{u_{2i-1}}=x_{u^*}+x_{u_{2i}}$.
Hence, $x_{u_{2i-1}}=\frac1{\rho^*-1}x_{u^*}$ and $x_{u_{2i-1}}+x_{u_{2i}}=\frac2{\rho^*-1}x_{u^*}<x_{u^*}$.
\end{proof}

According to Claim \ref{cl4.1}, if $A_1\neq\emptyset$, then
$$\sum_{u\in N_1(u^*)}x_u=\sum_{i\in A_1\cup A_2}(x_{u_{2i-1}}+x_{u_{2i}})<(|A_1|+2|A_2|)x_{u^*}=(e(N_1(u^*))+|A_2|)x_{u^*}.$$
Thus by (\ref{eq4.1}), we have
$$(m-d(u^*))x_{u^*}<(e(N_1({u^*}))+|A_2|+e(N^2({u^*}),N({u^*})))x_{u^*}.$$
This implies that $e(W)<|A_2|$, which contradicts (\ref{eq4.4}).
Hence, $A_1=\emptyset$, that is, $N_1({u^*})=\cup_{i\in A_2}\{u_{2i-1},u_{2i}\}$.
By (\ref{eq4.1}), we have
\begin{eqnarray}\label{eq4.5}
(m-d(u^*))x_{u^*}\leq\sum_{i\in A_2}(x_{u_{2i-1}}+x_{u_{2i}})+\sum_{w\in N^2({u^*})}d_{N({u^*})}(w)x_w
\leq (2|A_2|+e(N^2({u^*}),N({u^*})))x_{u^*}.
\end{eqnarray}
Notice that $|A_2|=e(N(u^*))$. (\ref{eq4.5}) implies that $e(W)\leq|A_2|$.
Combining with (\ref{eq4.4}), we have $e(W)=|A_2|$.
Now, both (\ref{eq4.4}) and (\ref{eq4.5}) become equalities.
Therefore, $\phi(u_{2i-1})+\phi(u_{2i})=2$ for each $i\in A_2$,
and $x_u=x_{u^*}$ for any vertex $u\in N_1({u^*})\cup N^2({u^*})$.
Note that $x_u\leq\frac2{\rho^*}x_{u^*}<x_{u^*}$ for any vertex $u$ of degree two.
It follows that $d(u)\geq3$ for any vertex $u\in N_1({u^*})\cup N^2({u^*})$.
If $A_2\neq \emptyset$, then for any $i\in A_2$,
both $u_{2i-1}$ and $u_{2i}$ have neighbors in $W$ and each of their neighbors is of degree at least 3.
Hence $\phi(u_{2i-1})+\phi(u_{2i})\geq4$,
a contradiction. Therefore, $A_2=\emptyset$ and $e(W)=|A_2|=0$.
\end{proof}

\begin{lem} \label{le4.4}
If $9\leq m\leq12$, then $e(W)=0$.
\end{lem}

\begin{figure}[t]
\centering \setlength{\unitlength}{2pt}
\begin{center}
\unitlength 1mm 
\linethickness{0.4pt}
\ifx\plotpoint\undefined\newsavebox{\plotpoint}\fi 
\begin{picture}(82.873,44.303)(0,0)
\multiput(41.57,39.553)(-.070035461,-.0336879433){564}{\line(-1,0){.070035461}}
\put(2.07,20.553){\line(1,0){28.5}}
\multiput(30.57,20.553)(.0337243402,.0549853372){341}{\line(0,1){.0549853372}}
\multiput(42.07,39.303)(.0336990596,-.0595611285){319}{\line(0,-1){.0595611285}}
\put(52.82,20.303){\line(1,0){28.5}}
\multiput(81.32,20.303)(-.0691489362,.0336879433){564}{\line(-1,0){.0691489362}}
\put(2.07,20.053){\line(0,-1){18.25}}
\put(2.07,1.803){\line(1,0){28.75}}
\multiput(30.82,1.803)(.0400728597,.0336976321){549}{\line(1,0){.0400728597}}
\put(81.57,20.303){\line(0,-1){18.25}}
\put(81.57,2.053){\line(-1,0){28.75}}
\multiput(52.82,2.053)(-.0400179856,.0337230216){556}{\line(-1,0){.0400179856}}
\put(41.57,38.803){\circle*{4}}
\put(1.82,20.303){\circle*{4}}
\put(1.82,1.803){\circle*{4}}
\put(80.82,20.303){\circle*{4}}
\put(81.07,2.303){\circle*{4}}
\put(52.82,2.303){\circle*{4}}
\put(52.32,20.553){\circle*{4}}
\put(30.57,20.803){\circle*{4}}
\put(30.57,2.053){\circle*{4}}
\put(41.32,44.303){\makebox(0,0)[cc]{$u^*$}}
\end{picture}
\end{center}
\caption{\footnotesize{The extremal graph $G^*$ with $e(G^*)=12$, $t=e(W)=2$ and $e(N^2(u^*),N(u^*))=|N_1^2(u^*)|=4$.}}\label{fig4.2}
\end{figure}
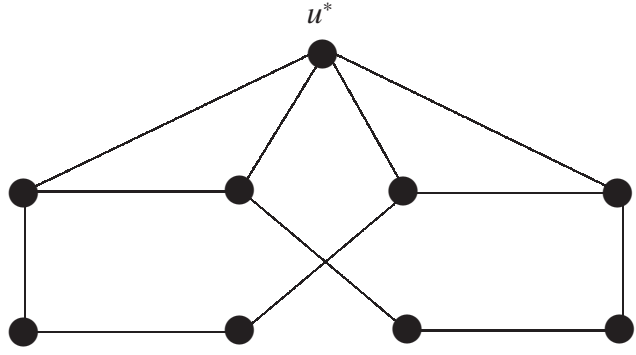

\begin{proof}
Let $t$ be the number of triangles in $G^*[N[u^*]]$.
It follows from (\ref{eq4.3}) that
$$e(W)\leq e(N_1({u^*}))=t,$$
and if $e(W)=e(N_1({u^*}))$, then $x_u=x_{u^*}$ for any vertex $u\in N_1({u^*})\cup N^2({u^*})$.
By Lemma \ref{le2.2}, $d(u)\geq 2$ for any $u\in N_1({u^*})\cup N^2({u^*})$.
Assume that $W\neq\emptyset$. First, we show that $e(W)\leq1$.

If $d(u_0)=2$ for some $u_0\in N_1({u^*})\cup N^2({{u^*}})$,
then $\rho^* x_{u_0}=\sum_{u\in N(u_0)}x_u\leq 2x_{u^*}$ and $x_{u_0}\leq\frac2{\rho^*}x_{u^*}<x_{u^*}$
(since $\rho^*\geq\sqrt{m}\geq3$).
Thus we have $e(W)<e(N_1({u^*}))$, that is, $e(N_1({u^*}))\geq e(W)+1$.
Therefore, $$e(G^*)\geq 3t+e(W)+e(N^2({u^*}),N({u^*}))\geq 4e(W)+3+|N^2({u^*})|.$$
Note that $e(G^*)\leq 12$. Thus $e(W)\leq 2$, and if $e(W)=2$ then $|N^2({u^*})|=1$.
Now we can find pendant vertices in $W$,
which contradicts Lemma \ref{le2.2}. Hence, $e(W)\leq1$.

Now assume that $d(u)\geq3$ for any $u\in N_1({u^*})\cup N^2({u^*})$.
However, $d_{N[{u^*}]}(u)=2$ for any $u\in N_1(u^*)$ and $d_{N_1({u^*})}(w)=1$ for any $w\in N_1^2({u^*}).$
Hence, $|N_1^2({u^*})|\geq |N_1({u^*})|=2t$. Furthermore,
\begin{eqnarray}\label{eq4.6}
e(G^*)\geq 3t+e(W)+e(N^2({u^*}),N({u^*}))\geq 4e(W)+|N_1^2({u^*})|\geq6e(W).
\end{eqnarray}
If $e(W)\geq2$, then all the inequalities in (\ref{eq4.6}) must be equalities.
This implies that $e(G^*)=12$, $t=e(W)=2$ and $e(N^2({u^*}),N({u^*}))=|N_1^2({u^*})|=4$
(see Fig.\ref{fig4.2}).
But now, all the vertices in $N_1^2({u^*})$ are of degree two, a contradiction.
Thus we also have $e(W)\leq1$.

Suppose that $e(W)=1$ and $w_1w_2$ is the unique edge of $G^*[W]$.
Recall that $d(w)\geq 2$ for any $w\in W$ and $d_{N(u^*)}(w)=1$ for any $w\in N_1^2(u^*)\subseteq W$.
Hence, $W\setminus\{w_1,w_2\}\subseteq N_0^2(u^*)$.
If $w_1,w_2\notin N_1^2(u^*)$,
then $d(u)=2$ for any $u\in N_1(u^*)$.
If $w_i\in N_1^2(u^*)$ for some $i\in \{1,2\}$,
then $d(w_i)=2$, and $d(u)\neq2$ only holds for its unique neighbor $u\in N_1(u^*)$.
This implies that there are at least $|N_1(u^*)|$ vertices of degree two in $N_1(u^*)\cup \{w_1,w_2\}$.
Note that $x_u\leq \frac2{\rho^*}x_{u^*}$ for any vertex $u$ of degree two.
Thus by (\ref{eq4.1}), we have
\begin{eqnarray}\label{eq4.7}
(m-d({u^*}))x_{u^*}\leq\left(2e(N_1(u^*))+e(N^2(u^*), N(u^*))-(1-\frac2{\rho^*})|N_1(u^*)|\right)x_{u^*}.
\end{eqnarray}
Notice that $|N_1(u^*)|=2e(N_1(u^*))$ and $\rho^*\geq3$.
(\ref{eq4.7}) implies that $e(W)\leq \frac13e(N_1(u^*))=\frac t3.$
It follows that $m>3t+e(W)\geq10e(W)=10$ and hence $\rho^*\geq\sqrt{m}>3$.
Again by (\ref{eq4.7}) we have $e(W)<\frac13e(N_1(u^*)).$
Thus, $t\geq 3e(W)+1$ and hence $m>3t+e(W)>12$, a contradiction.
Therefore, $e(W)=0$, as desired.
\end{proof}

In the following, we give the proof of Theorem \ref{th1.3}~(ii).

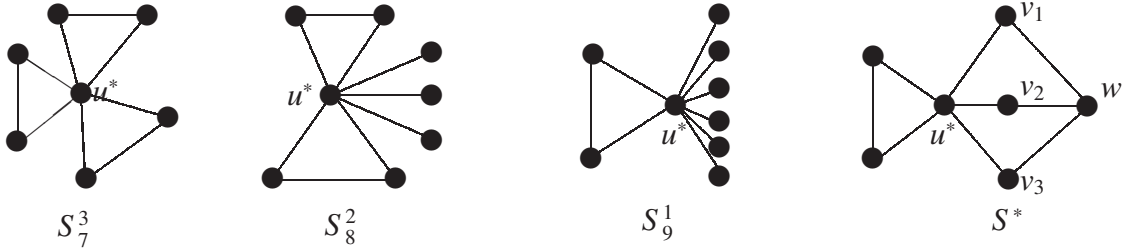
\begin{figure}[t]
\centering \setlength{\unitlength}{2.4pt}
\begin{center}
\unitlength 0.55mm 
\linethickness{0.4pt}
\ifx\plotpoint\undefined\newsavebox{\plotpoint}\fi 
\begin{picture}(263.664,53.414)(0,0)
\put(1.664,42.75){\line(0,-1){22.25}}
\put(1.664,20.5){\line(5,4){15}}
\put(16.664,32.5){\line(-3,2){15}}
\multiput(16.664,32.5)(-.033653846,.123397436){156}{\line(0,1){.123397436}}
\put(11.414,51.75){\line(1,0){21.5}}
\multiput(32.914,51.75)(-.0336842105,-.0405263158){475}{\line(0,-1){.0405263158}}
\multiput(16.914,32.5)(.0333333,-.6916667){30}{\line(0,-1){.6916667}}
\multiput(17.914,11.75)(.0460674157,.0337078652){445}{\line(1,0){.0460674157}}
\multiput(38.414,26.75)(-.115591398,.033602151){186}{\line(-1,0){.115591398}}
\put(1.664,42.25){\circle*{5}}
\put(1.414,21.25){\circle*{5}}
\put(18.164,12.25){\circle*{5}}
\put(37.914,27){\circle*{5}}
\put(11.414,52){\circle*{5}}
\put(32.914,51.75){\circle*{5}}
\put(16.914,32.75){\circle*{5}}
\put(69.414,51.75){\line(1,0){21.25}}
\put(90.664,51.75){\line(0,1){.25}}
\multiput(90.664,52)(-.0337150127,-.0502544529){393}{\line(0,-1){.0502544529}}
\multiput(77.414,32.25)(-.0336734694,.0775510204){245}{\line(0,1){.0775510204}}
\multiput(69.164,51.25)(-.03125,.09375){8}{\line(0,1){.09375}}
\multiput(77.414,32.25)(-.0337209302,-.0476744186){430}{\line(0,-1){.0476744186}}
\put(62.914,11.75){\line(1,0){30.25}}
\multiput(93.164,11.75)(-.0336956522,.0456521739){460}{\line(0,1){.0456521739}}
\put(77.664,32.75){\line(0,1){0}}
\put(77.664,32.75){\line(0,-1){.25}}
\multiput(77.664,32.5)(.0793269231,.0336538462){312}{\line(1,0){.0793269231}}
\put(102.414,43){\line(0,1){0}}
\put(77.664,32.25){\line(1,0){24.75}}
\multiput(77.914,32)(.078125,-.0337171053){304}{\line(1,0){.078125}}
\put(101.664,21.5){\circle*{5}}
\put(101.664,32.25){\circle*{5}}
\put(101.664,42.75){\circle*{5}}
\put(68.914,51.5){\circle*{5}}
\put(90.164,51.75){\circle*{5}}
\put(63.164,12.25){\circle*{5}}
\put(92.914,12.25){\circle*{5}}
\put(77.164,32.25){\circle*{5}}
\put(140.164,42.5){\line(0,-1){25.25}}
\multiput(140.164,17.25)(.0524611399,.0336787565){386}{\line(1,0){.0524611399}}
\multiput(160.414,30.25)(-.0575842697,.0337078652){356}{\line(-1,0){.0575842697}}
\multiput(160.164,29.75)(.0337243402,.0659824047){341}{\line(0,1){.0659824047}}
\multiput(160.664,30.25)(.0337078652,-.0512640449){356}{\line(0,-1){.0512640449}}
\multiput(160.914,30)(.0336391437,.0397553517){327}{\line(0,1){.0397553517}}
\multiput(160.914,30)(.09375,.033482143){112}{\line(1,0){.09375}}
\multiput(160.914,29.75)(.0352564103,-.0336538462){312}{\line(1,0){.0352564103}}
\multiput(160.664,30)(.076241135,-.033687943){141}{\line(1,0){.076241135}}
\put(140.664,42.25){\circle*{5}}
\put(140.164,17){\circle*{5}}
\put(160.664,30){\circle*{5}}
\put(171.164,52){\circle*{5}}
\put(171.164,43){\circle*{5}}
\put(171.164,34){\circle*{5}}
\put(171.164,26){\circle*{5}}
\put(171.164,19.75){\circle*{5}}
\put(171.164,12.75){\circle*{5}}
\put(208.414,41.75){\line(0,-1){25.25}}
\multiput(208.664,16.25)(.0438931298,.0337150127){393}{\line(1,0){.0438931298}}
\multiput(225.914,29.5)(-.0484550562,.0337078652){356}{\line(-1,0){.0484550562}}
\put(208.414,41.75){\circle*{5}}
\put(208.164,17){\circle*{5}}
\put(225.664,30){\circle*{5}}
\multiput(225.414,30)(.0336956522,.0472826087){460}{\line(0,1){.0472826087}}
\put(225.414,30){\line(1,0){15.75}}
\multiput(226.164,29.5)(.0337078652,-.0393258427){445}{\line(0,-1){.0393258427}}
\put(240.414,51.5){\circle*{5}}
\put(240.914,30){\circle*{5}}
\put(241.164,12){\circle*{5}}
\put(239.414,29.75){\line(1,0){21.5}}
\multiput(260.664,29.75)(-.0380539499,-.0337186898){519}{\line(-1,0){.0380539499}}
\multiput(260.914,29.5)(-.0336938436,.0361896839){601}{\line(0,1){.0361896839}}
\put(260.164,29.75){\circle*{5}}
\put(225.414,23){\makebox(0,0)[cc]{$u^*$}}
\put(247,52.5){\makebox(0,0)[cc]{$v_1$}}
\put(247,33.5){\makebox(0,0)[cc]{$v_2$}}
\put(247,10.5){\makebox(0,0)[cc]{$v_3$}}
\put(266,33.5){\makebox(0,0)[cc]{$w$}}
\put(160.414,23){\makebox(0,0)[cc]{$u^*$}}
\put(70,32.5){\makebox(0,0)[cc]{$u^*$}}
\put(23,33.5){\makebox(0,0)[cc]{$u^*$}}
\put(15.164,0){\makebox(0,0)[cc]{$S_7^3$}}
\put(79.664,.5){\makebox(0,0)[cc]{$S_8^2$}}
\put(156.414,1.5){\makebox(0,0)[cc]{$S_9^1$}}
\put(240.914,2.25){\makebox(0,0)[cc]{$S^*$}}
\end{picture}
\end{center}
\caption{\footnotesize{Possible extremal graphs for $m=9$.}}\label{fig4.3}
\end{figure}

\begin{proof}
From Lemmas \ref{le4.3} and \ref{le4.4},
we have $e(W)=0$.
If $N_1({u^*})=\emptyset$, then $G^*$ is triangle-free.
By Theorem \ref{th1.1},
$G^*$ is a complete bipartite graph, as desired.

Next, assume that $N_1(u^*)\neq\emptyset$. Then $e(N_1(u^*))\geq1$.
Note that $d(w)\geq2$ and $d_{N({u^*})}(w)=1$ for any $w\in N_1^2({u^*})$.
Then $d_W(w)\geq1$ for any $w\in N_1^2({u^*})$.
However $e(W)=0$, thus, $N_1^2({u^*})=\emptyset$ and $W=N_0^2({u^*}).$
Hence, $d(u_{2i-1})=d(u_{2i})=2$ for each edge $u_{2i-1}u_{2i}$ of $G^*[N_1({u^*})]$.
Similarly as above, we have
$x_{u_{2i-1}}=x_{u_{2i}}=\frac1{\rho^*-1}x_{u^*}$.
By (\ref{eq4.1}), we have
\begin{eqnarray}\label{eq4.8}
(m-d(u^*))x_{u^*}\leq\sum_{u\in N_1({u^*})}x_u+\sum_{w\in N_0^2({u^*})}d_{N({u^*})}(w)x_w
\leq\left(\frac{2e(N_1({u^*}))}{\rho^*-1}+e(N_0^2({u^*}),N({u^*}))\right)x_{u^*}.
\end{eqnarray}
It follows that $e(W)\leq0$, since $m\geq9$ and $\rho^*\geq\sqrt{m}\geq3$.
This implies that $e(W)=0$ and (\ref{eq4.8}) is an equality.
Thus, $m=9$ and $x_w=x_{u^*}$ for any vertex $w\in N_0^2(u^*)$.

Given any $w\in N_0^2(u^*)$,
$\rho^*x_{u^*}=\rho^*x_w=\sum_{u\in N(w)}x_u\leq d(w)x_{u^*}$.
It follows that $d(w)\geq\rho^*\geq3$.
Hence $|N_0({u^*})|\geq3$ and $e(G^*)\geq 3e(N_1({u^*}))+|N_0({u^*})|+3|N_0^2({u^*})|$.
Since $e(G^*)=9$, we have $|N_0^2({u^*})|\leq1$.
If $|N_0^2({u^*})|=1$, then $e(N_1({u^*}))=1$, $|N_0({u^*})|=3$ and
$G^*\cong S^*$ (see Fig.\ref{fig4.3}).
Note that $d(v_i)=2$ for each neighbor $v_i$ of $w$.
Then, $x_{v_i}\leq \frac2{\rho^*} x_{u^*}<x_{u^*}$
and hence $x_w=\frac1{\rho^*}(x_{v_1}+x_{v_2}+x_{v_3})<x_{u^*}$.
This implies that (\ref{eq4.8}) is a strict inequality,
that is, $e(W)<0$, a contradiction.
Therefore, $N_0^2({u^*})=\emptyset$, i.e., $W=\emptyset$.
Then $G^*$ is isomorphic to one of $S_7^3$, $S_8^2$ and $S_9^1$
(see Fig.\ref{fig4.3}).
This completes the proof.
\end{proof}

\section{Proof of Theorem \ref{th1.4}}\label{se4}

In this section, we characterize the extremal graph $G^*$ for $C_5$-free or $C_6$-free graphs.
The following first lemma was proposed as a conjecture by
Guo, Wang and Li \cite{GUO} and confirmed by Sun and Das \cite{SUN} recently.

\begin{lem}{\rm(\cite{SUN})}\label{le5.1}
Let $v\in V(G)$ with $d_G(v)\geq1$. Then $\rho(G)\leq\sqrt{\rho^2(G-v)+2d_G(v)-1}$.
The equality holds if and only if either $G\cong K_n$ or $G\cong K_{1,n-1}$ with $d_G(v)=1$.
\end{lem}

\begin{lem}{\rm(\cite{PA})}\label{le5.2}
Let $G$ be a connected graph,
and $X=(x_{u_1},x_{u_2},\ldots,x_{u_n})^T$ be its eigenvector corresponding to $\rho=\rho(G)$ with $\|X\|_p=1$.
Then for $1\leq p<\infty$, the maximum coordinate $x_{u_1}\leq (\frac{\rho^{p-2}}{1+\rho^{p-2}})^{\frac1p}$.
\end{lem}

Now we come back to consider the desired extremal graph $S_{\frac{m+3}2,2}$
in Theorem \ref{th1.4}.

\begin{lem}\label{le5.3}
Let $\rho^\star=\rho(S_{\frac{m+3}2,2})$. Then
$\rho^\star$ is the largest root of $\rho^2-\rho-(m-1)=0$.
And hence, $\rho^\star=\frac{1+\sqrt{4m-3}}2.$
\end{lem}

\begin{proof}
Let $X$ be the Perron vector of $S_{\frac{m+3}2,2}$.
Let $u_1,u_2$ be the dominating vertices,
and $v_1,v_2,\ldots,v_{\frac{m-1}2}$ be the remaining vertices of degree two in $S_{\frac{m+3}2,2}$.
By symmetry, we have
$$\rho^\star x_{u_1}=x_{u_1}+\frac{m-1}2x_{v_1}, ~~\rho^\star x_{v_1}=2x_{u_1}.$$
Clearly, ${\rho^\star}^2-\rho^\star-(m-1)=0$.
\end{proof}

Let $R_k$ denote the graph obtained from $k$ copies of $K_4$ by sharing a vertex
and $H_{t,s}=\langle T,S\rangle$ be a bipartite graph with $|T|=t$ and $|S|=s$
(note that $H_{t,s}$ is not necessarily complete bipartite).
Let $H_{t,s}\circ R_{k}$ denote the graph obtained
by joining the dominating vertex $u$ of $R_k$ with the independent set $T$ of $H_{t,s}$,
for $k\geq1$ and $t, s\geq0$ (see Fig.\ref{fig5.1}).

\begin{figure}[t]
\centering \setlength{\unitlength}{3.0pt}
\begin{center}
\unitlength 4.2mm 
\linethickness{0.4pt}
\begin{picture}(18.252,12.042)(0,0)
\put(5.215,10.562){\circle*{.6}}
\put(5.171,8.618){\circle*{.6}}
\put(5.171,6.673){\circle*{.6}}
\put(5.215,2.873){\circle*{.6}}
\put(1.193,3.094){\circle*{.6}}
\put(1.149,9.899){\circle*{.6}}
\put(1.149,7.38){\circle*{.6}}
\multiput(5.215,10.651)(-.1730938,-.0331456){24}{\line(-1,0){.1730938}}
\multiput(5.259,10.562)(-.042813106,-.033605986){96}{\line(-1,0){.042813106}}
\multiput(5.259,10.607)(-.033689001,-.063031035){122}{\line(0,-1){.063031035}}
\multiput(5.127,8.574)(-.1016466,.03314563){40}{\line(-1,0){.1016466}}
\multiput(5.171,8.618)(-.10024532,-.03341511){41}{\line(-1,0){.10024532}}
\multiput(5.259,8.662)(-.033587572,-.047022601){125}{\line(0,-1){.047022601}}
\multiput(5.127,6.673)(-.041892394,.033605986){96}{\line(-1,0){.041892394}}
\multiput(10.209,7.557)(-.05597929,.03339115){90}{\line(-1,0){.05597929}}
\multiput(5.082,8.529)(.1815118,-.0331456){28}{\line(1,0){.1815118}}
\multiput(10.165,7.601)(-.1737288,-.0335266){29}{\line(-1,0){.1737288}}
\multiput(10.209,7.601)(-.034440977,-.033526615){145}{\line(-1,0){.034440977}}
\put(17.192,11.579){\line(0,-1){3.005}}
\multiput(17.192,8.574)(-.04999013,.03332675){61}{\line(-1,0){.04999013}}
\multiput(14.142,10.607)(.0997933,.0327892){31}{\line(1,0){.0997933}}
\multiput(14.186,10.518)(-.046203,-.03364784){88}{\line(-1,0){.046203}}
\put(10.12,7.557){\line(0,1){.177}}
\put(10.12,7.734){\line(0,-1){.2652}}
\multiput(10.209,7.557)(.057820711,.033513915){120}{\line(1,0){.057820711}}
\multiput(10.165,7.557)(.2209709,.0331456){32}{\line(1,0){.2209709}}
\put(17.192,1.591){\line(0,1){3.049}}
\multiput(17.192,4.64)(-.092406,-.0334804){33}{\line(-1,0){.092406}}
\multiput(14.142,3.536)(.05257583,-.03352661){58}{\line(1,0){.05257583}}
\multiput(17.192,1.591)(-.039256277,.033577696){179}{\line(-1,0){.039256277}}
\multiput(10.165,7.601)(.033513915,-.034987054){120}{\line(0,-1){.034987054}}
\multiput(10.076,7.646)(.076859433,-.033626002){92}{\line(1,0){.076859433}}
\put(5.7,11.379){\makebox(0,0)[cc]{$v_1$}}
\put(5.7,9.171){\makebox(0,0)[cc]{$v_2$}}
\put(5.7,7.359){\makebox(0,0)[cc]{$v_3$}}
\put(6.031,2.74){\makebox(0,0)[cc]{$v_t$}}
\put(0,9.96){\makebox(0,0)[cc]{$w_1$}}
\put(.044,7.292){\makebox(0,0)[cc]{$w_2$}}
\put(0,3.2){\makebox(0,0)[cc]{$w_s$}}
\put(10.209,7.646){\circle*{.6}}
\put(14.186,10.518){\circle*{.6}}
\put(17.147,11.623){\circle*{.6}}
\put(17.147,8.662){\circle*{.6}}
\put(17.103,4.596){\circle*{.6}}
\put(14.098,3.491){\circle*{.6}}
\put(17.192,1.635){\circle*{.6}}
\put(10.165,8.75){\makebox(0,0)[cc]{$u$}}
\put(14.01,11.5){\makebox(0,0)[cc]{$u_1$}}
\put(18.2,11.888){\makebox(0,0)[cc]{$u_2$}}
\put(18.254,8.795){\makebox(0,0)[cc]{$u_3$}}
\put(17.1,6.894){\makebox(0,0)[cc]{\Large$\vdots$}}
\put(1.149,5.215){\makebox(0,0)[cc]{$\vdots$}}
\put(5.259,4.861){\makebox(0,0)[cc]{$\vdots$}}
\end{picture}
\end{center}
\vspace{-0.3cm}
\caption{The graph $H_{t,s}\circ R_{k}$, where $k\geq1$ and $t, s\geq0$.}\label{fig5.1}
\end{figure}
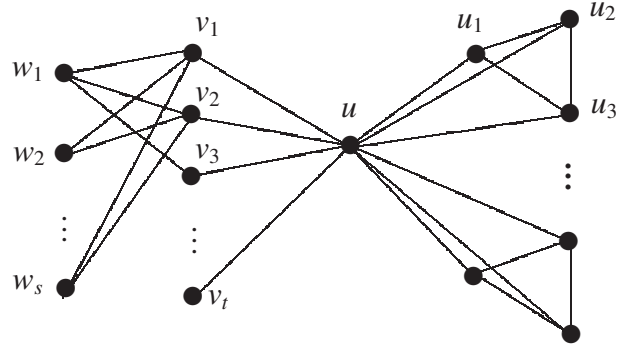

\begin{lem}\label{le5.4}
If $k\geq1$ and $m=6k+t\geq8$, then $\rho(H_{t,0}\circ R_k)<\frac{1+\sqrt{4m-3}}2$.
\end{lem}

\begin{proof}
Let $\rho=\rho(H_{t,0}\circ R_k)$, and $X$ be the Perron vector of $H_{t,0}\circ R_k$.
Let $u$ be the dominating vertex,
$u_1,u_2,\ldots,u_{3k}$ be the remaining vertices of $R_k$,
and $v_1,v_2,\ldots,v_t$ be the vertices in $H_{t,0}$.
By symmetry, we have
$$\rho x_{v_1}=x_u, ~~\rho x_u=3kx_{u_1}+tx_{v_1}, ~~\rho x_{u_1}=x_u+2x_{u_1}.$$
Solving this system, we have
\begin{eqnarray}\label{eq5.1}
f(\rho)=\rho^3-2\rho^2-(3k+t)\rho+2t=0.
\end{eqnarray}
Let $\rho^\star=\frac{1+\sqrt{4m-3}}2$. Then
\begin{eqnarray}\label{eq5.2}
{\rho^\star}^2=\rho^\star+(m-1)=\rho^\star+(6k+t-1).
\end{eqnarray}
Substituting (\ref{eq5.2}) to (\ref{eq5.1}),
we have $f(\rho^\star)=(3k-2)\rho^\star-(6k-t-1)$.
Since $\rho^\star>3$ for $m\geq8$, we have $f(\rho^\star)>3k+t-5\geq0$.
Moreover, (\ref{eq5.2}) implies that
$f'(\rho)=3\rho^2-4\rho-(3k+t)>0$ for $\rho\geq\rho^\star$.
Hence, $f(\rho)>0$ for $\rho\geq\rho^\star$.
Thus, $\rho(H_{t,0}\circ R_k)<\rho^\star$.
\end{proof}

\begin{rem}\label{re4.1}
If $m=7$, then $t=k=1$, $\rho^\star=3$ and $f(\rho^\star)=3k+t-5=-1$.
Hence, $\rho(H_{1,0}\circ R_1)>\rho^\star.$ Notice that $H_{1,0}\circ R_1$ is $C_5$-free.
The condition $m\geq 8$ is best possible for Theorem \ref{th1.4}~(i).
\end{rem}

We now characterize the extremal graph $G^*$ for $C_5$-free graphs.
Note that $G^*[N(u^*)]$ contains no paths of length $3$ for forbidding $C_5$.
Thus, we have the following result on the local structure of $G^*$.

\begin{lem} \label{le5.5}
Each connected component of $G^*[N(u^*)]$ is either a triangle or
a star $K_{1,r}$ for some $r\geq 0$, where $K_{1,0}$ is a singleton component.
\end{lem}

Notice that $S_{\frac{m+3}2,2}$ is $C_t$-free for $t\geq5$.
By Lemma \ref{le5.3}, we have
$\rho^*\geq \rho(S_{\frac{m+3}2,2})=\frac{1+\sqrt{4m-3}}2$.
Recall that $N_0(u^*)=\{u~|~u\in N(u^*), d_{N(u^*)}(u)=0\}$ and $W=V(G^*)\setminus N[{u^*}]$.
For convenience, let $W_0=N_W(N_0(u^*))$, i.e., $\cup_{u\in N_0(u^*)}N_W(u)$.
The following lemma gives a clearer local structure of the extremal graph $G^*$.

\begin{lem}\label{le5.6}
$e(W)=0$ and $W=W_0$.
\end{lem}

\begin{proof}
According to (\ref{eq2.1}) and (\ref{eq2.2}),
we have
\begin{eqnarray}\label{eq5.3}
({\rho^*}^2-\rho^*)x_{u^*}=
|N({u^*})|x_{u^*}+\sum_{u\in {N(u^*)\setminus N_0(u^*)}}(d_{N({u^*})}(u)-1)x_u-\sum_{u\in N_0(u^*)}x_u+\sum_{w\in W}d_{N({u^*})}(w)x_w.
\end{eqnarray}
Let $N_+(u^*)=N(u^*)\setminus N_0(u^*)$. Then
\begin{eqnarray}\label{eq5.4}
({\rho^*}^2-\rho^*)x_{u^*}\leq \left(|N({u^*})|+2e(N_+(u^*))-|N_+(u^*)|-\sum_{u\in N_0(u^*)}\frac {x_u}{x_{u^*}}+e(W,N({u^*}))\right)x_{u^*}
\end{eqnarray}
Note that $\rho^*\geq\frac{1+\sqrt{4m-3}}2$,
that is, ${\rho^*}^2-\rho^*\geq m-1$.
Combining with (\ref{eq5.4}), we have
\begin{eqnarray}\label{eq5.5}
e(W)\leq e(N_+(u^*))-|N_+(u^*)|-\sum_{u\in N_0(u^*)}\frac {x_u}{x_{u^*}}+1,
\end{eqnarray}
since $e(N(u^*))=e(N_+(u^*))$.
According to (\ref{eq5.3}) and (\ref{eq5.4}),
the equality in (\ref{eq5.5}) holds
if and only if ${\rho^*}^2-\rho^*=m-1$, $x_w=x_{u^*}$ for any $w\in W$ and $u\in N_+(u^*)$ with $d_{N(u^*)}(u)\geq2$.
By Lemma \ref{le5.5},
$e(N_+(u^*))\leq |N_+(u^*)|$ and hence $e(W)\leq1.$

Assume that $e(W)=1$. Let $w_iw_j$ be the unique edge in $G^*[W]$.
By Lemma \ref{le2.2}, $d(w)\geq2$ for each $w\in W$.
This implies that both $w_i$ and $w_j$ have neighbors in $N({u^*})$.
Note that $G^*$ is $C_5$-free. Thus, $w_i,w_j$ must have the same and the unique neighbor $u$ in $N({u^*})$.
Now $u$ is a cut vertex, which contradicts Lemma \ref{le2.2}.
Hence, $e(W)=0.$

Suppose that $W\setminus W_0\neq\emptyset$. For the reason of $C_5$-free,
we can see that $d(w)\leq2$ (and hence $d(w)=2$) for each $w\in W\setminus W_0$.
By Lemma \ref{le2.2},
all the vertices in $W\setminus W_0$ have the same neighborhood, say, $\{u_1,u_2\}$.
For avoiding pentagons,
$u_1$ and $u_2$ are in the same component $H$ of $G^*[N({u^*})]$ and $H$ can only be a copy of $K_{1,1}.$
Let $G=G^*-\{wu_1~|~w\in W\setminus W_0\}+\{wu^*~|~w\in W\setminus W_0\}$.
Clearly, $G$ is $C_5$-free.
Note that $x_{u^*}\geq x_{u_1}$ and $N_{G^*}(u^*)\varsubsetneqq N_{G}(u^*)$.
By Lemma \ref{le2.1}, $\rho(G)>\rho^*$, a contradiction.
Hence, $W=W_0$.
\end{proof}

\begin{lem}\label{le5.7}
Either $G^*\cong S_{\frac{m+3}{2}, 2}$ or $G^*\cong H_{t,s}\circ R_k$
for some $k\geq1$ and some bipartite graph $H_{t,s}$ of size $m-6k-t$.
\end{lem}

\begin{proof}
By Lemma \ref{le5.6} and (\ref{eq5.5}),
we have
\begin{eqnarray}\label{eq5.6}
e(N_+(u^*))\geq|N_+(u^*)|+\sum_{u\in N_0(u^*)}\frac {x_u}{x_{u^*}}-1.
\end{eqnarray}

Firstly, assume that $N_0(u^*)\neq\emptyset$.
Then $\sum_{u\in N_0(u^*)}\frac {x_u}{x_u^*}>0$ and hence
$e(N_+(u^*))\geq|N_+(u^*)|$.
It follows from Lemma \ref{le5.5} that
$G^*[N_+(u^*)]$ consists of disjoint union of $k$ triangles.
If $k=0$, then $G^*$ is triangle-free.
By Theorem \ref{th1.1},
$\rho^*\leq\sqrt{m}<\frac{1+\sqrt{4m-3}}2$, a contradiction.
Thus, $k\geq1$. By Lemma \ref{le5.6},
$G^*\cong H_{|N_0(u^*)|, |W_0|}\circ R_k$ for some bipartite graph $H_{|N_0(u^*)|,|W_0|}$.

Secondly, assume that $N_0(u^*)=\emptyset$.
Then $W_0=\emptyset$,
and by Lemma \ref{le5.6}, $W=\emptyset$.
Thus, (\ref{eq5.5}) becomes
$e(N_+(u^*))-|N_+(u^*)|\geq-1$.
Let $c$ be the number of star-components of $G^*[N_+(u^*)]$.
Then $e(N_+(u^*))-|N_+(u^*)|=-c$. It follows that $c\leq1$.
If $c=0$, then $G^*\cong R_k$ for $k=\frac m6$.
By Lemma \ref{le5.4},
$\rho^*<\frac{1+\sqrt{4m-3}}2$, a contradiction.
Thus, $c=1$. Let $k$ be the number of triangle-components of $G^*[N_+(u^*)]$.
If $k\geq1$, then $G^*$ is isomorphic to $K_{1,r}\bullet R_{k}$ for some $r\geq1$
(see Fig.\ref{fig5.2}).
By symmetry, $\rho^* x_{u_1}=2x_{u_1}+x_{u^*}$.
Since $m\geq 8$ and $\rho^*\geq\frac{1+\sqrt{4m-3}}2>3$,
we have $x_{u_1}=\frac{x_{u^*}}{\rho^*-2}<{x_{u^*}}$.
This implies that (\ref{eq5.5}) is strict.
Correspondingly, $$e(W)<e(N_+(u^*))-|N_+(u^*)|+1=-c+1=0,$$ a contradiction.
Hence, $k=0$. Now $G^*[N(u^*)]\cong K_{1,r}$ and $G^*$ is obtained by joining $u^*$ with each vertex of $K_{1,r}$.
Thus, $G^*\cong S_{\frac{m+3}{2}, 2}$.
This completes the proof.
\end{proof}

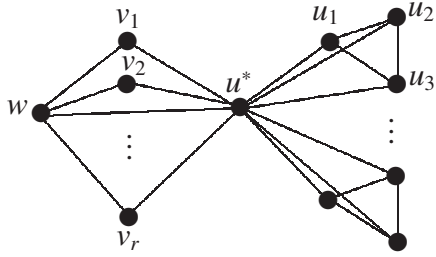
\begin{figure}[t]
\centering \setlength{\unitlength}{3.0pt}
\begin{center}
\unitlength 3mm 
\linethickness{0.4pt}
\ifx\plotpoint\undefined\newsavebox{\plotpoint}\fi 
\begin{picture}(80,11.07)(7,0)
\put(28.815,9.591){\circle*{.815}}
\multiput(28.859,9.591)(-.0428125,-.033604167){96}{\line(-1,0){.0428125}}
\multiput(33.809,6.585)(-.05597778,.0334){90}{\line(-1,0){.05597778}}
\put(40.792,10.607){\line(0,-1){3.005}}
\multiput(40.792,7.602)(-.05,.03332787){61}{\line(-1,0){.05}}
\multiput(37.742,9.635)(.0998065,.0327742){31}{\line(1,0){.0998065}}
\multiput(37.786,9.546)(-.04620455,-.03364773){88}{\line(-1,0){.04620455}}
\put(33.72,6.585){\line(0,1){.177}}
\put(33.72,6.763){\line(0,-1){.265}}
\multiput(33.809,6.585)(.057816667,.033516667){120}{\line(1,0){.057816667}}
\multiput(33.765,6.585)(.2209688,.0331875){32}{\line(1,0){.2209688}}
\put(40.792,.619){\line(0,1){3.049}}
\multiput(40.792,3.669)(-.0924242,-.0334545){33}{\line(-1,0){.0924242}}
\multiput(37.742,2.565)(.05258621,-.03355172){58}{\line(1,0){.05258621}}
\multiput(40.792,.619)(-.039256983,.033581006){179}{\line(-1,0){.039256983}}
\multiput(33.765,6.63)(.033508333,-.034983333){120}{\line(0,-1){.034983333}}
\multiput(33.676,6.674)(.076858696,-.033630435){92}{\line(1,0){.076858696}}
\put(28.859,10.607){\makebox(0,0)[cc]{$v_1$}}
\put(33.809,6.674){\circle*{.815}}
\put(37.786,9.546){\circle*{.815}}
\put(40.747,10.651){\circle*{.815}}
\put(40.747,7.69){\circle*{.815}}
\put(40.703,3.624){\circle*{.815}}
\put(37.698,2.52){\circle*{.815}}
\put(40.792,.664){\circle*{.815}}
\put(33.765,7.78){\makebox(0,0)[cc]{$u^*$}}
\put(37.61,10.741){\makebox(0,0)[cc]{$u_1$}}
\put(41.852,10.916){\makebox(0,0)[cc]{$u_2$}}
\put(41.894,7.824){\makebox(0,0)[cc]{$u_3$}}
\put(40.52,5.923){\makebox(0,0)[cc]{$\vdots$}}
\put(28.859,0.75){\makebox(0,0)[cc]{$v_r$}}
\put(28.783,7.703){\circle*{.815}}
\put(29.1,8.5){\makebox(0,0)[cc]{$v_2$}}
\put(28.871,1.781){\circle*{.815}}
\multiput(33.809,6.718)(-.1582438,.0327893){31}{\line(-1,0){.1582438}}
\put(28.903,7.734){\line(-1,0){.1326}}
\multiput(33.853,6.63)(-.035050578,-.03352664){145}{\line(-1,0){.035050578}}
\multiput(24.705,6.276)(.835672,.032141){11}{\line(1,0){.835672}}
\put(28.9,5.243){\makebox(0,0)[cc]{$\vdots$}}
\put(23.937,6.63){\makebox(0,0)[cc]{$w$}}
\multiput(24.881,6.364)(.09470187,.03367178){42}{\line(1,0){.09470187}}
\multiput(24.881,6.364)(.03351394,-.037933361){120}{\line(0,-1){.037933361}}
\put(24.982,6.377){\circle*{.815}}
\end{picture}
\end{center}
\vspace{-0.3cm}
\caption{$K_{1,r}\bullet R_{k}$ for $r\geq1$ and $k\geq1$.}\label{fig5.2}
\end{figure}

In the following, we give the proof of Theorem \ref{th1.4}~(i).

\begin{proof}
Suppose that $G^*\cong H_{t,s}\circ R_k$ for some $k\geq1$ and some bipartite graph $H_{t,s}=<T,S>$
(see Fig.\ref{fig5.1}).
By the symmetry, $\rho^*x_{u_i}=2x_{u_i}+x_u$ for each vertex $u_i$ in any $K_4$-block of $G^*$.
Recall that $\rho^*>3$. Thus, $x_{u_i}=\frac{x_u}{\rho^*-2}<x_u$.
If $x_v\geq x_u$ for some vertex $v\neq u$, then $v\in T\cup S$.
Thus we can select a block $B\cong K_4$ and
define $G=G^*+\{vj~|~j\in N_B(u)\}-\{uj~|~j\in N_B(u)\}$.
Clearly, $G$ is $C_5$-free.
And by Lemma \ref{le2.1}, we have $\rho(G)>\rho^*$,
a contradiction. It follows that $x_v<x_u$ for any vertex $v\neq u$.

Notice that $\rho^*\geq \frac{1+\sqrt{4m-3}}2$.
Lemma \ref{le5.4} implies that $s\geq1$.
If $s=1$, say $S=\{w_1\}$ and $v_1\in N(w_1)$, then we define $G'=G^*+uw_1-v_1w_1$.
Clearly, $G'$ is $C_5$-free.
Since $x_{v_1}<x_u$, we have $\rho(G')>\rho^*$, a contradiction.
Thus, $s\geq2$.
By Lemma \ref{le2.2}, $d(w_i)\geq 2$ for any $w_i\in S$.
Hence, $t\geq2$ and $e(T,S)\geq 2s$.

Without loss of generality, we may assume that $x_{v_1}$, $x_{v_2}$ attain the largest
two coordinates among all vertices in $T$ and $w_1$ attains the maximum degree among all vertices in $S$
(see Fig.\ref{fig5.1}).
Since $G^*$ is an extremal graph, all the vertices in $S$ are adjacent to $v_1$ and $v_2$.
This implies that $N(v_1)=N(v_2)$ and hence $x_{v_1}=x_{v_2}.$
If $e(T,S)=2s$, then $d(w_1)=2$ and $\rho^*x_{w_1}=2x_{v_1}$.
If $e(T,S)=2s+1$, then $d(w_1)=3$ and $\rho^*x_{w_1}=x_{v_1}+x_{v_2}+x_{v_3}\leq3x_{v_1}$,
where $v_3\in N(w_1)\setminus \{v_1,v_2\}$.
In both cases, $x_{w_1}<x_{v_1}$, since $\rho^*>3$.
Define $G''=G^*+v_1v_2-w_1v_2$ if $e(T,S)=2s$ or $G''=G^*+v_1v_2-w_1v_3$ if $e(T,S)=2s+1$.
One can see that $G''$ is $C_5$-free and $\rho(G'')>\rho^*,$
a contradiction. Thus, $e(T,S)\geq2s+2$.
It follows that $t\geq3$ and $m=6k+t+e(T,S)\geq6k+2s+5.$
Let $\{u,u_1,u_2,u_3\}$ be a 4-clique of $G^*$.
Next, we use induction on $m$ to show $\rho^*<\frac{1+\sqrt{4m-3}}2$.

Firstly, assume that $k=1.$ Recall that $s\geq2$. Thus, $m\geq15$.
By Lemma \ref{le5.1}, we have
$${\rho^*}^2\leq \rho^2(G^*-u_1)+5\leq\rho^2(G^*-u_1-u_2)+8\leq\rho^2(G^*-u_1-u_2-u_3)+9.$$
Note that $G^*-u_1-u_2-u_3$ is a bipartite graph of size $m-6$.
By Theorem \ref{th1.1}, we have
$\rho^2(G^*-u_1-u_2-u_3)\leq m-6$.
Hence, ${\rho^*}^2\leq m+3<(\frac{1+\sqrt{4m-3}}2)^2$, since $m\geq15$.

Secondly, assume that $k\geq 2.$ Then $m\geq21$.
Note that $x_{u_1}=x_{u_2}=x_{u_3}$ (see Fig.\ref{fig5.1}).
Thus $\rho^* x_{u_1}=2x_{u_1}+x_u$.
By Lemma \ref{le5.2},
$x_u\leq\frac1{\sqrt2}$ and thus
\begin{eqnarray}\label{eq5.7}
x_{u_1}^2=(\frac {x_u}{\rho^*-2})^2\leq\frac1{2(\rho^*-2)^2}.
\end{eqnarray}
Let $G'''=G^*-\{u_iu_j~|~1\leq i<j\leq3\}$. We have
$$\rho(G''')\geq {X^*}^T(A(G'''))X^*={X^*}^T(A(G^*))X^*-6x_{u_i}x_{u_j}=\rho^*-6x_{u_1}^2$$
and hence
\begin{eqnarray*}
\rho(G''')(\rho(G''')-1)\geq(\rho^*-6x_{u_1}^2)(\rho^*-6x_{u_1}^2-1)>{\rho^*}^2-\rho^*-6(2\rho^*-1)x_{u_1}^2.
\end{eqnarray*}
Combing with (\ref{eq5.7}), we have
\begin{eqnarray}\label{eq5.8}
\rho(G''')(\rho(G''')-1)>{\rho^*}^2-\rho^*-\frac{3(2\rho^*-1)}{(\rho^*-2)^2}.
\end{eqnarray}
Note that $e(G''')=m-3$ and $G'''$ belongs to the class of $H_{t+3, s}\circ R_{k-1}.$
On one hand, by induction hypothesis, $\rho(G''')<\frac{1+\sqrt{4m-15}}2$, that is, $\rho(G''')(\rho(G''')-1)<m-4.$
On the other hand, recall that $\rho^*\geq\frac{1+\sqrt{4m-3}}2$ and $m\geq21$.
We have $\rho^*\geq5$ and hence $\frac{3(2\rho^*-1)}{(\rho^*-2)^2}\leq3$.
Thus by (\ref{eq5.8}),
${\rho^*}^2-\rho^*<m-1,$ that is, $\rho^*<\frac{1+\sqrt{4m-3}}2$.

In both cases, we get a contradiction: $\rho^*<\frac{1+\sqrt{4m-3}}2$.
Thus, $G^*\ncong H_{t,s}\circ R_k$.
According to Lemma \ref{le5.7},
$G^*\cong S_{\frac{m+3}{2}, 2}$. This completes the proof.
\end{proof}

In what follows, we consider $C_6$-free graphs.
Recall that $W=V(G^*)\setminus N[{u^*}]$. Let $W_H=N_W(V(H))$ for any component $H$ of $G^*[N(u^*)]$.
Since $G^*$ is $C_6$-free, $U_{H_i}\cap U_{H_j}=\emptyset$ for any two component $H_i$ and $H_j$ of $G^*[N(u^*)]$,
unless one is an isolated vertex and the other is a star.
Note that $\rho^*\geq \rho(S_{\frac{m+3}2,2})=\frac{1+\sqrt{4m-3}}2,$ that is, ${\rho^*}^2-\rho^*\geq m-1.$
And $\sum_{w\in W}d_{N({u^*})}(w)x_w\leq e(W,N(u^*))x_{u^*}$.
Combining with (\ref{eq5.3}), we have
$$\left(m-1-|N(u^*)|-e(W,N(u^*))+\sum_{u\in N_0(u^*)}\frac{x_u}{x_{u^*}}\right)x_{u^*}\leq \sum_{u\in {N(u^*)\setminus N_0(u^*)}}\left(d_{N({u^*})}(u)-1\right)x_u.$$
Define $\gamma(H)=\sum_{u\in V(H)}(d_{H}(u)-1)x_u$ for any non-trivial component $H$ of $G^*[N(u^*)]$.
Then
\begin{eqnarray}\label{eq6.1}
\left(e(N(u^*))+e(W)+\sum_{u\in N_0(u^*)}\frac{x_u}{x_{u^*}}-1\right)x_{u^*}\leq \sum_{H}\gamma(H),
\end{eqnarray}
where $H$ takes over all non-trivial components of $G^*[N(u^*)]$.

\begin{lem} \label{le6.1}
$G^*[N(u^*)]$ contains no any cycle of length four or more.
\end{lem}

\begin{proof}
Since $G^*$ is $C_6$-free,
$G^*[N(u^*)]$ does not contain any path of length 4 and any cycle of length more than 4.
Moreover, if $G^*[N(u^*)]$ is of circumference 4,
then components containing $C_4$ are isomorphic to $C_4,$ $C_4+e$ or $K_4$;
if $G^*[N(u^*)]$ is of circumference 3,
then components containing $C_3$ are isomorphic to $K_{1,r}+e$ for $r\geq2$.
Let $\mathcal{H}$ be the family of components of $G^*[N(u^*)]$ which contain $C_4$ and
$\mathcal{H'}$ be the family of other components of $G^*[N(u^*)]$.
Then $\mathcal{H'}$ consists of trees or unicyclic components. Hence,
\begin{eqnarray}\label{eq6.2}
\gamma(H)\leq(2e(H)-|H|)x_{u^*}\leq e(H)x_{u^*},
\end{eqnarray}
for any $H\in \mathcal{H'}$.
Moreover, note that $G^*$ is $C_6$-free.
Thus, $W_{H_1}\cap W_{H_2}=\emptyset$ and $e(W_{H_1},W_{H_2})=0$ for any $H_1,H_2\in\mathcal{H}$ with $H_1\neq H_2.$
We can see that
\begin{eqnarray}\label{eq6.3}
e(W)\geq\sum_{H\in\mathcal{H}}e(W_H,W).
\end{eqnarray}
Combining (\ref{eq6.2}), (\ref{eq6.3}) with
(\ref{eq6.1}) and abandoning the item $\sum_{u\in N_0(u^*)}\frac{x_u}{x_{u^*}}$,
we have
$$\left(\sum_{H\in \mathcal{H}}(e(H)+e(W_H,W))-1\right)x_{u^*}\leq \sum_{H\in \mathcal{H}}\gamma(H).$$
Suppose that $\mathcal{H}\neq\emptyset$.
In order to get a contradiction,
it suffices to show $$\gamma(H)<(e(H)+e(W_H,W)-1)x_{u^*}$$ for each $H\in\mathcal{H}$.
Let $H^*\in\mathcal{H}$ with $V(H^*)=\{u_1,u_2,u_3,u_4\}$.
First, assume that $W_{H^*}=\emptyset$.
Let $x_{u_{i^*}}=\max\{x_{u_i}~|~1\leq i\leq4\}$.
Then $\rho^* x_{u_{i^*}}=\sum_{u\in N(u_{i^*})}x_u\leq x_{u^*}+3x_{u_{i^*}}$.
Note that $\rho^*\geq\frac{1+\sqrt{4m-3}}2>5$, since $m>21$.
Hence $x_{u_{i^*}}\leq\frac{x_{u^*}}{\rho^*-3}<\frac{1}2 x_{u^*}$ and
$\gamma(H^*)\leq(2e(H^*)-|H^*|)x_{u_{i^*}}<(e(H^*)-2)x_{u^*},$
as desired.
Next assume that $W_{H^*}\neq\emptyset$.
We consider two cases.

Firstly, there are $k\geq2$ vertices of $W_{H^*}$ with mutual distinct neighbor in $V(H^*)$.
Since $G^*$ is $C_6$-free, $d_{N(u^*)}(w)=1$ for any $w\in W_{H^*}$, and $N_{H^*}(w)=N_{H^*}(w')$
for each pair of adjacent vertices $w,w'\in W_{H^*}$.
Let $w_1,w_2,\ldots,w_k\in W_{H^*}$ with mutual distinct neighbor in $V(H^*)$.
Then $\{w_1,w_2,\ldots,w_k\}$ is an independent set.
By Lemma \ref{le2.2},
$d(w_i)\geq2$ for $i\in\{1,2,\ldots,k\}$ and $d(w_i)=2$ for at most one $w_i$.
Thus, $\sum_{1\leq i\leq k}d(w_i)\geq 3k-1$.
This implies that $e(W_{H^*},W)\geq2k-1$ and hence
\begin{eqnarray}\label{eq6.4}
e(H^*)\leq e(K_4)\leq e(W_{H^*},W)-2k+7\leq e(W_{H^*},W)+3.
\end{eqnarray}
Note that $|H^*|=4$.
By (\ref{eq6.4}), we have
\begin{eqnarray}\label{eq6.5}
\gamma(H^*)=\sum_{u_i\in V(H^*)}(d_{H^*}(u_i)-1)x_{u_i}\leq(2e(H^*)-|H^*|)x_{u^*}\leq (e(H^*)+e(W_{H^*},W)-1)x_{u^*}.
\end{eqnarray}
If (\ref{eq6.5}) holds in equality,
then (\ref{eq6.4}) also holds in equality.
This implies that $k=2$, $H^*\cong K_4$ and $x_{u_i}=x_{u^*}$ for each $u_i\in V(H^*)$.
Since $k=2$, we can find a vertex $u_i\in V(H^*)$ with $N_W(u_i)=\emptyset$.
Thus, $\rho^* x_{u_i}=\sum_{u\in N(u_i)}x_{u}\leq4x_{u^*}$.
That is, $x_{u_i}\leq\frac{4}{\rho^*}x_{u^*}<x_{u^*}.$
Thus, (\ref{eq6.5}) is strict, as desired.

Secondly, all vertices in $W_{H^*}$ have a common neighbor, say $u_1$, in $V(H^*)$.
Now $N_W(u_i)=\emptyset$ for $i\in\{2,3,4\}$. Let $x_{u_2}=\max\{x_{u_i}~|~i=2,3,4\}$.
Thus, $\rho^* x_{u_2}\leq(x_{u^*}+x_{u_1})+(x_{u_3}+x_{u_4})\leq2(x_{u^*}+x_{u_2})$,\
i.e., $x_{u_2}\leq\frac2{\rho^*-2}x_{u^*}<\frac{2}3x_{u^*}$.
Note that $d_{H^*}(u_1)\leq3$. By the definition of $\gamma(H^*)$, we have
\begin{eqnarray*}
\gamma(H^*) &\leq& (d_{H^*}(u_1)-1)x_{u_1}+(2e(H^*)-d_{H^*}(u_1)-3)x_{u_2}\\
&<& \left((d_{H^*}(u_1)-1)+(\frac43e(H^*)-\frac23 d_{H^*}(u_1)-2)\right)x_{u^*}\\
&\leq& (\frac43e(H^*)-2)x_{u^*}.
\end{eqnarray*}
It follows that $\gamma(H^*)<e(H^*)x_{u^*}$, since $\frac13 e(H^*)\leq2$.
Moreover, note that $d(w)\geq2$ for any $w\in W_{H^*}$.
Thus, $e(W_{H^*},W)\geq1$ and hence $\gamma(H^*)<(e(H^*)+e(W_{H^*},W)-1)x_{u^*}$.
\end{proof}

Lemma \ref{le6.1} implies that
each component of $G^*[N(u^*)]$ is a tree or a unicyclic graph $K_{1,r}+e$.
Let $c$ be the number of non-trivial tree-components of $G^*[N(u^*)]$.
Then
$$\sum_{H}\gamma(H)\leq \sum_{H}(2e(H)-|H|)x_{u^*}=(e(N(u^*))-c)x_{u^*},$$
where $H$ takes over all non-trivial components of $G^*[N(u^*)]$.
Thus, (\ref{eq6.1}) becomes
$$\left(e(W)+\sum_{u\in N_0(u^*)}\frac{x_u}{x_{u^*}}-1\right)x_{u^*}\leq -cx_{u^*}.$$
Equivalently,
\begin{eqnarray}\label{eq6.6}
e(W)\leq 1-c-\sum_{u\in N_0(u^*)}\frac{x_u}{x_{u^*}}.
\end{eqnarray}

\begin{lem} \label{le6.2}
$e(W)=0$ and $G^*[N(u^*)]$ contains no triangle.
\end{lem}

\begin{proof}
By (\ref{eq6.6}), we have $e(W)\leq1$,
and if $e(W)=1$, then $c=0$ and $N_0(u^*)=\emptyset$. Suppose that $e(W)=1$.
Then, each component $H$ of $G^*[N(u^*)]$ is isomorphic to $K_{1,r}+e$ for some $r\geq2$.
Let $w_1w_2$ be the unique edge in $G^*[W]$.
For avoiding hexagons, $w_1,w_2$ belong to the same $W_H$ and
they must have a unique and common neighbor $u\in V(H)$.
This leads to a pendant triangle and a cut vertex $u$,
which contradicts Lemma \ref{le2.2}.
Hence, $e(W)=0$.

Suppose to contrary that $G^*[N(u^*)]$ contains triangles, that is,
$G^*[N(u^*)]$ contains a component $H^*\cong K_{1,r}+e$.
We have the following claims.

\begin{cla} \label{cl6.1}
$H^*\ncong C_3$.
\end{cla}

\begin{proof} Assume that $H^*\cong C_3$ and $V(H^*)=\{u_1,u_2,u_3\}$.
If $W_{H^*}=\emptyset$, then $x_{u_1}=x_{u_2}=x_{u_3}$ and $\rho^* x_{u_1}=x_{u^*}+2x_{u_1}.$
Thus, $x_{u_1}=\frac {x_{u^*}}{\rho^*-2}<\frac13 x_{u^*},$ since $\rho^*>5$.
Therefore, $$\gamma(H^*)=\sum_{1\leq i\leq3}(d_{H^*}(u_i)-1)x_{u_i}=3x_{u_1}<(e(H^*)-2)x_{u^*}.$$
Note that $e(W)=0$, and by (\ref{eq6.2}),
$\gamma(H)\leq e(H)x_{u^*}$ for any other component $H$.
Thus, (\ref{eq6.1}) becomes an impossible inequality:
$$\left(e(N(u^*))+\sum_{u\in N_0(u^*)}\frac{x_u}{x_{u^*}}-1\right)x_{u^*}\leq \sum_H\gamma(H)<(e(N(u^*))-2)x_{u^*}.$$
Hence, $W_{H^*}\neq\emptyset.$
Clearly, $2\leq d(w)\leq |H^*|=3$ for any $w\in W_{H^*}$.
If there exists $w\in W_{H^*}$ with $d(w)=3$,
then $W_{H^*}=\{w\}$ (otherwise we will get a hexagon).
If $d(w)=2$ for any $w\in W_{H^*}$, then by Lemma \ref{le2.2},
all the vertices in $W_{H^*}$ have common neighborhood, say $\{u_1,u_2\}$.
In both cases, we define $G=G^*+\{wu^*~|~w\in W_{H^*}\}-\{wu_1~|~w\in W_{H^*}\}$.
Clearly, $G$ is $C_6$-free.
Moreover, since $x_{u^*}\geq x_{u_1}$,
by Lemma \ref{le2.1}, $\rho(G)>\rho^*$, a contradiction.
Thus, the claim holds.
\end{proof}

\begin{cla} \label{cl6.2}
$W_{H^*}=\emptyset$ and $H^*$ is the unique non-trivial component of $G^*[N(u^*)]$.
\end{cla}

\begin{proof}
Claim \ref{cl6.1} implies that $H^*\cong K_{1,r}+e$ for some $r\geq3$.
Since $G^*$ is $C_6$-free, $d_{N(u^*)}(w)=1$ for any $w\in W_{H^*}$.
Note that $e(W)=0$. Thus, $d(w)=1$ for any $w\in W_{H^*}$,
which contradicts Lemma \ref{le2.2}.
Hence, $W_{H^*}=\emptyset$.

Now $V(H^*)$ contains $r-2$ vertices of degree two.
By Lemma \ref{le2.2},
all the non-adjacent vertices of degree two have common neighborhood.
Thus, there exists no vertices of degree two out of $H^*$.
Hence, $H^*$ is the unique component which contains triangles.

Suppose that $G^*[N(u^*)]$ contains a non-trivial tree-component $H$,
i.e., $c\geq1$.
By (\ref{eq6.6}), $N_0(u^*)=\emptyset$ and $c=1$.
Since no vertices of degree two out of  $H^*$,
we have $W_{H}\neq\emptyset$ (otherwise, each leaf of $H$ is of degree two in $G^*$),
and $d(w)\geq3$ for any $w\in W_{H}$.
Note that $e(W)=0$ and $G^*$ is $C_6$-free.
We can see that $H\cong K_{1,2}$ and $|W_H|=1$.
Let $W_{H}=\{w\}$ and $u$ be the central vertex of $H$.
Clearly, $G^*+wu^*-wu$ is $C_6$-free.
And by Lemma \ref{le2.1},
$\rho(G^*+uu^*-uv)>\rho^*$, a contradiction.
The claim holds.
\end{proof}

\begin{cla} \label{cl6.3}
Let $W_0=N_W(N_0(u^*)).$ Then $W_0=\emptyset$.
\end{cla}

\begin{proof}
Suppose that $W_0\neq\emptyset$. Then $N_0(u^*)\neq\emptyset$.
By Lemma \ref{le2.2}, $d(w)\neq1$ for any $w\in W_0$.
And since no vertices of degree two out of $H^*$, $d(w)\neq2$ for any $w\in N_0(u^*)\cup W_0$.

If $d(u)\geq3$ for some $u\in N_0(u^*)$,
then $u$ has at least two neighbors $w_1,w_2\in W_0$ with $\min \{d(w_1), d(w_2)\}\geq3$.
Note that $e(W)=0$ and $e(W_0, V(H^*))=0$ for forbidding hexagons.
Thus, there exist two vertices $u_1,u_2\in N_0(u^*)\setminus\{u\}$ with $u_i\in N(w_i)$.
Then $u^*u_1w_1uw_2u_2u^*$ is a $C_6$ (see Fig.\ref{fig6.1}),
a contradiction. Thus, $d(u)=1$ for any $u\in N_0(u^*)$,
that is, $W_0=\emptyset.$
\end{proof}

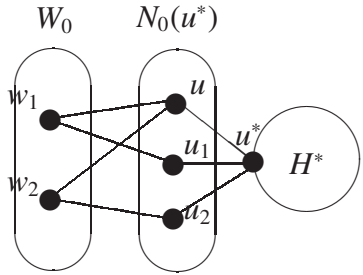
\begin{figure}[h]
\centering \setlength{\unitlength}{3.0pt}
\begin{center}
\unitlength 1.85mm 
\ifx\plotpoint\undefined\newsavebox{\plotpoint}\fi 
\begin{picture}(26.652,18.184)(0,0)
\put(2.733,8.093){\oval(5.466,15.977)[]}
\put(11.668,7.989){\oval(5.466,15.977)[]}
\put(2.523,10.931){\circle*{1.5}}
\put(2.523,5.15){\circle*{1.5}}
\put(11.562,12.087){\circle*{1.5}}
\put(11.352,7.673){\circle*{1.5}}
\put(11.352,3.784){\circle*{1.5}}
\multiput(2.418,11.036)(.22817006,.03332821){41}{\line(1,0){.22817006}}
\multiput(11.773,12.403)(-.043676697,-.03359746){219}{\line(-1,0){.043676697}}
\multiput(11.562,7.778)(-.096443017,.033592511){97}{\line(-1,0){.096443017}}
\multiput(11.457,3.784)(-.20544628,.03344474){44}{\line(-1,0){.20544628}}
\put(20.968,8.093){\circle{9.366}}
\put(17.133,7.883){\circle*{1.5}}
\put(11.562,12.403){\line(5,-4){5.781}}
\put(17.343,7.778){\line(-1,0){6.0965}}
\multiput(11.247,3.889)(.053462164,.03352712){116}{\line(1,0){.053462164}}
\put(.631,12.403){\makebox(0,0)[cc]{$w_1$}}
\put(.631,6.096){\makebox(0,0)[cc]{$w_2$}}
\put(13.139,8.829){\makebox(0,0)[cc]{$u_1$}}
\put(13.139,4.099){\makebox(0,0)[cc]{$u_2$}}
\put(13.139,13.254){\makebox(0,0)[cc]{$u$}}
\put(16.8,9.88){\makebox(0,0)[cc]{$u^*$}}
\put(21,7.778){\makebox(0,0)[cc]{$H^*$}}
\put(2.838,18.079){\makebox(0,0)[cc]{$W_0$}}
\put(11.667,18.184){\makebox(0,0)[cc]{$N_0(u^*)$}}
\end{picture}
\end{center}
\caption{The extremal graph $G^*$ in Claim $3$.}\label{fig6.1}
\end{figure}

According to Claims \ref{cl6.1}-\ref{cl6.3}, $W=\emptyset$,
and $G^*$ is obtained by joining $u^*$ with $H^*\cong K_{1,r}+e$ and $|N_0(u^*)|$ isolated vertices.
Let $u_1,u_2,u_3\in V(H^*)$ with $d_{H^*}(u_1)=d_{H^*}(u_2)=2$ and $d_{H^*}(u_3)=r$.
Then $x_{u_1}=x_{u_2}$ and $\rho^*x_{u_1}=x_{u_2}+x_{u_3}+x_{u^*}\leq x_{u_1}+2x_{u^*}$.
Hence, $x_{u_1}\leq\frac{2}{\rho^*-1}x_{u^*}<\frac{1}2 x_{u^*},$ since $\rho^*>5.$
Thus, $$\gamma(H^*)=(r-1)x_{u_3}+x_{u_1}+x_{u_2}\leq(r-1)x_{u^*}+2x_{u_1}<rx_{u^*}=(e(H^*)-1)x_{u^*}.$$
Then, (\ref{eq6.1}) becomes an impossible inequality:
$$\left(e(H^*)+\sum_{u\in N_0(u^*)}\frac{x_u}{x_{u^*}}-1\right)x_{u^*}<(e(H^*)-1)x_{u^*}.$$
This completes the proof.
\end{proof}

In the following, we give the proof of Theorem \ref{th1.4}~(ii).

\begin{proof}
By Lemma \ref{le6.1} and
Lemma \ref{le6.2},
$G^*[N(u^*)]$ consists of tree-components.
(\ref{eq6.6}) implies that $c\leq1$,
and if $c=1$ then $N_0(u^*)=\emptyset$.
Assume that $c=0$. Then $G^*$ is bipartite, since $e(W)=0$.
By Theorem \ref{th1.1},
$\rho^*\leq\sqrt{m}<\frac{1+\sqrt{4m-3}}2$, a contradiction.
Hence, $c=1$ and $N_0(u^*)=\emptyset$. Let $H$ be the unique component of $G^*[N(u^*)]$.

Since $G^*$ is $C_6$-free, $diam(H)\leq3$.
If $diam(H)=3$, then $H$ is a double star.
For forbidding $C_6$, $d_{N(u^*)}(w)=1$ for any $w\in W_H$.
This implies $W_H=\emptyset$ (otherwise, $d(w)=1$ since $e(W)=0$).
However, $H$ contains two non-adjacent vertices $u_1,u_2$ with
$d_{G^*}(u_1)=d_{G^*}(u_2)=2$ and $N(u_1)\neq N(u_2)$.
It contradicts Lemma \ref{le2.2}.
Thus, $diam(H)\leq2$, i.e., $H$ is a star $K_{1,r}$.
Let $V(H)=\{u_0,u_1,\ldots,u_r\}$, where $u_0$ is the center vertex of $H$.
Note that $\rho^*>5$ and $\rho^* x_{u^*}=\sum_{0\leq i\leq r}x_{u_i}\leq (r+1)x_{u^*}$.
Hence, $r>4$.

If there exists some $w\in W_H$ with at least two neighbors,
say $u_1,u_2$, in $V(H)\setminus \{u_0\}$,
then $u^*u_1wu_2u_0u_3u^*$ is a hexagon, a contradiction.
It follows that $d(w)=2$ and $u_0\in N(w)$ for any $w\in W_H$.
By Lemma \ref{le2.2},
all the vertices of $W_H$ have common neighbors, say $u_0,u_1$.
However, $N(u_2)=\{u^*,u_0\}\neq N(w)$ for $w\in W_H$.
This contradicts Lemma \ref{le2.2}.
Hence, $W_H=\emptyset$ and $G^*$ is the join of $u^*$ with a star $H$.
Clearly, $G^*\cong S_{\frac{m+3}2,2}.$ This completes the proof.
\end{proof}

\section{Concluding remarks}\label{se5}

In this section, we first consider the existence of cycles with consecutive lengths.
We need introduce a well-known result due to Erd\H{o}s and Gallai.

\begin{lem}{(\cite{EG})}\label{7.1}
For every $k\geq0$, $ex(n, P_{k+1})\leq \frac12(k-1)n$, with equality if and only if $n=kt$,
the extremal graph is the disjoint union of $K_k$'s.
\end{lem}

In the following, we give the proof of Theorem \ref{th1.5}.

\begin{proof}
Let $G$ be a graph of size $m$ and $Y$ be an eigenvector of $A(G)$ corresponding to $\rho=\rho(G)$ with
$\sum_{i=1}^{|G|}y_i=1$. Define $f(A(G))=A^2(G)-\frac12(2k-1)A(G)$, where $k$ is a positive integer.
Then
$f(A(G))Y=f(\rho)Y$. Thus,
\begin{eqnarray}\label{eq7.1}
f(\rho)=\sum\limits_{i=1}^{|G|}f(\rho)y_i=\sum\limits_{i=1}^{|G|}\left(f(A(G))Y\right)_i
=\sum\limits_{i=1}^{|G|}\left(\sum\limits_{j=1}^{|G|}f_{ij}y_j\right)=
\sum\limits_{j=1}^{|G|}\left(\sum\limits_{i=1}^{|G|}f_{ij}\right)y_j=\sum\limits_{j=1}^{|G|}f_j(G)y_j, \end{eqnarray}
where $f_{ij}$ is the $(i,j)$-element of $f(A(G))$ and $f_j(G)$ is the sum of the $j$-th column of $f(A(G))$.
For any $j\in V(G)$, let $U_j=V(G)\setminus N[j]$, and $g_j(G)$ be the sum of the $j$-th column of $A^2(G)$.
Clearly, $$g_j(G)=\sum_{i\in N(j)}d_G(i)=d_G(j)+2e(N(j))+e(N(j),U_j).$$
Furthermore, $$f_j(G)=d_G(j)+2e(N(j))+e(N(j),U_j)-\frac12(2k-1)d_G(j)\leq m+e(N(j))-\frac12(2k-1)d_G(j).$$
If $\rho>\frac{k-\frac12+\sqrt{4m+(k-\frac12)^2}}2$, then $f(\rho)>m$.
(\ref{eq7.1}) implies that there exists some $j^*\in V(G)$ such that $f_{j^*}(G)>m$.
Thus, $e(N(j^*))>\frac12(2k-1)d_G(j^*)$.
By Lemma \ref{7.1},
$G[N(j^*)]$ contains a copy of $P_{2k+1}$.
It follows that $G$ contains a cycle $C_t$ for every $t\leq 2k+2$.
\end{proof}

Theorems \ref{th1.1}-\ref{th1.2}
imply that if $\rho(G)\geq\sqrt{m}$, then $G$ contains $C_3$ and
$C_4$ unless $G$ is star $S_{m+1,1}$.
From Theorem \ref{th1.4}
we know that if $\rho(G)\geq\frac{1+\sqrt{4m-3}}2$,
then $G$ contains $C_t$ for every $t\leq 6$ unless $G$ is a book $S_{\frac{m+3}2, 2}$.
These results inspire us to look for a more general spectral condition for cycles with consecutive lengths,
which is stated in the following conjecture.

\begin{conj}\label{co6.1}
Let $k$ be a fixed positive integer and $G$ be a graph of sufficiently large size $m$ without isolated vertices.
If $\rho(G)\geq\frac{k-1+\sqrt{4m-k^2+1}}2$,
then $G$ contains a cycle of length $t$ for every $t\leq 2k+2$,
unless $G\cong S_{\frac mk+\frac{k+1}2, k}$.
\end{conj}

On the other hand, we know that $\rho(G)\leq\sqrt{m}$ for several kinds of graphs, such as,
bipartite graphs, $C_3$-free graphs, $C_4$-free graphs,
$K_{2,r+1}$-free graphs, $\{C_3^+,C_4^+\}$-free graphs, and so on.
A natural question: how large can a graph family be such that $\rho\leq\sqrt{m}$?
An \emph{$r$-book}, denoted by $B_r$, is the graph obtained from $r$ triangles by sharing one edge.
At the end of this paper, we pose the following conjecture.

\begin{conj} \label{co6.2}
Let $G\in \mathbb{G}(m, B_{r+1})$ and $m\geq m_0(r)$ for large enough $m_0(r)$.
Then $\rho(G)\leq\sqrt m$, with equality if and only if $G$ is a complete bipartite graph.
\end{conj}

%

\vspace{6bp}
\noindent{\bf Acknowledgements}
We would like to show our great gratitude to anonymous referees for their valuable suggestions which largely improve the quality of this paper.

\end{document}